\documentclass[runningheads,a4paper]{llncs}

\usepackage{graphicx}
\usepackage{graphics,latexsym,xspace}
\usepackage[all]{xy}
\usepackage{gnuplottex}
\usepackage{lscape}
\usepackage{amssymb}
\usepackage{verbatim}
\usepackage{amsmath}
\usepackage{amsfonts}
\usepackage{tikz}
\newcommand{\ben}{\begin{enumerate}}
\newcommand{\bit}{\begin{itemize}}
\newcommand{\een}{\end{enumerate}}
\newcommand{\eit}{\end{itemize}}

\newcommand{\un}{\underline}

\newcommand{\mc}{\mathcal}

\newcounter{romc}

\newcounter{alphc}
\newcommand{\bd}{\begin{definition}}
\newcommand{\ed}{\end{definition}}
\newcommand{\bp}{\begin{proposition}}
\newcommand{\ep}{\end{proposition}}
\newcommand{\ba}{\begin{eqnarray*}}
\newcommand{\ea}{\end{eqnarray*}}
\newcommand{\bde}{\begin{description}}
\newcommand{\ede}{\end{description}}
\newcommand{\be}{\begin{eqnarray}}
\newcommand{\ee}{\end{eqnarray}}
\newcommand{\bn}{\begin{note}}
\newcommand{\en}{\end{note}}

\newtheorem{notation}{Notation}

\newcommand{\blr}{\begin{list}{~(\roman{romc})~}
{\usecounter{romc}
        \setlength{\topsep}{0pt} \setlength{\itemsep}{0pt}}}
\newcommand{\elr}{\end{list}}
\newcommand{\bla}{\begin{list}{~(\alph{alphc})~} {\usecounter{alphc}
        \setlength{\topsep}{0pt} \setlength{\itemsep}{0pt}}}
\newcommand{\ela}{\end{list}}
\begin{document}
\title{Stone and double Stone algebras: Boolean  and Rough Set Representations, 3-valued and 4-valued Logics}
\author{Arun Kumar \inst{1}  }
%
\institute{$^{1}$Department of Mathematics,\\ Institute of Science, Banaras Hindu University, 221005, India\\
\email{arunk2956@gmail.com}
}

\maketitle

\begin{abstract}
Moisil in 1941, while constructing the algebraic models of n-valued {\L}ukasiewicz logic defined the set $B^{[n]}$,where $B$ is a Boolean algebra and `n'  being a natural number. Further it was proved by Moisil himself the representations of n-valued {\L}ukasiewicz Moisil algebra in terms of $B^{[n]}$. In this article,    
structural representation results for Stone, dual Stone and double Stone algebras are proved similar to Moisil's work by showing that elements of these algebras  can be looked upon as
monotone ordered tuple of sets.  3-valued semantics of logic for Stone algebra, dual Stone algebras and 4-valued semantics of logic for double Stone algebras are proposed and established soundness and completeness results.  \\ 

\end{abstract}
\section{Introduction}
The Stone's \cite{Stone36} representation theorem for Boolean algebra identifies an element of a given Boolean algebra as a set, and has vital role in algebra and logic.
It shows that algebraic, set theoretic   and \textbf{T}rue-\textbf{F}alse semantics of classical propositional logic are equivalent. On the other hand, there are some works where an element of a given algebra is identified by a pair of sets. Some well known examples are:
\begin{itemize}
\item  (Moisil (cf. \cite{Cignoli07})) Moisil represented each  element of a given 3-valued LM algebra $\mathcal{A}$ can be looked upon as a monotone ordered pair of sets.
\item Rasiowa \cite{Rasiowa1} represented De Morgan algebras as set-based De Morgan algebras, where De Morgan negation  is defined by an involution function.
\item In Dunn's \cite{Dunn66,Dunn99} representation, each element of a De Morgan algebra can be identified with an ordered pair of sets, where De Morgan negation is defined as reversing the order in the pair.  

\end{itemize} 
\noindent Note that all these representation results have their logical insight, e.g., Dunn's representation leads to  his well known 4-valued semantics of De Morgan logic. In \cite{KB17}, author provided a representation result for Kleene algebras which is very similar to Mosil's representation of 3-valued LM algebra, and studied 3-valued aspect of corresponding logic (for Kleene algebras). 

In this article, we prove structural theorems for Stone and dual Stone algebras similar to the  Moisil representation of 3-valued LM algebra. Let us define the algebras.
\begin{definition}
An algebra $\mathcal{S} := (S,\vee,\wedge,\sim,0,1)$ is a {\rm Stone algebra} if 
\begin{enumerate}
\item $(S,\vee,\wedge,\sim,0,1)$ is a bounded distributive pseudo complemented lattice, i.e, $\forall a \in S$,
 $\sim a  = max\{c\in S: a \wedge c = 0\}$ exists.
\item $\sim a \vee \sim\sim a = 1$, for all $a \in S$.

\end{enumerate}
\end{definition}
The axiom  $\sim a \vee \sim\sim a = 1$ first appeared in \textbf{Problem 70} - What is the most general pseudo-complemented distributive lattice in which $a^{*} \vee a^{**} = I$ identically? (attributed to M.H.Stone) of Birkhoff's book \cite{BIRK48}. The formal definition of Stone algebra (lattice) appeared in Gratzer and Schmidt \cite{Gratzer57} and have been extensively studied in literature (cf. \cite{Balbes74}).
\noindent The dual notion  of a given Stone algebra is known as dual Stone algebra. For the self explanatory of  this article, let us explicitly define the dual Stone algebra. 
\begin{definition}
An algebra $\mathcal{DS} := (DS,\vee,\wedge,\neg,0,1)$ is a {\rm dual Stone algebra} if 
\begin{enumerate}
\item $(DS,\vee,\wedge,\sim,0,1)$ is a bounded distributive dual pseudo complemented lattice, i.e, $\forall a \in DS$,
 $\sim a  = min\{c\in DS: a \vee c = 1\}$ exists.
\item $\neg a \wedge \neg\neg a = 0$, for all $a \in DS$ (dual Stone property).

\end{enumerate}
\end{definition} 
In this article, we also present a representation result for double Stone algebra, where each element of a given double Stone algebra is identified as a monotone ordered 3-tuple of sets. Double Stone algebra is a bounded distributive lattice which is both Stone and dual Stone algebra.
\begin{definition}
An algebra $\mathcal{A} := (A,\vee,\wedge,\sim,\neg, 0,1)$ is  a {\rm   double Stone algebra} if 
\begin{enumerate}
\item $(A,\vee,\wedge,\sim,0,1)$ is a bounded distributive  lattice,
\item $(A,\vee,\wedge,\sim,0,1)$ is a Stone algebra, 
\item $(A,\vee,\wedge,\neg,0,1)$ is dual Stone algebra,
\end{enumerate}
\end{definition}

\noindent It is well known that 3-valued LM algebras are algebraic models of 3-valued {\L}ukasiewicz logic, and Kleene algebras, Stone algebras, dual Stone algebras and double Stone algebras appear as reduct algebras \cite{Boicescu91}  of 3-valued LM algebras. So, it is natural to ask `can we provide multi valued semantics of logics corresponding these algebras?'. In \cite{KB17}, author provided a 3-valued semantics of logic corresponding to Kleene algebras. In this paper we provide 3-valued semantics of logic ($\mathcal{L}_{\mathcal{S}}$) for Stone,  logic ($\mathcal{L}_{\mathcal{DS}}$) for dual Stone algebras and a 4-valued semantics of logic ($\mathcal{L}_{\mathcal{DBS}}$) for double Stone algebras.

In other aspect of this paper, we make explicit connections between rough sets, Stone and dual Stone algebras by providing rough set representations of these algebras.   Rough set theory, introduced by Pawlak \cite{pawlak82,Pawlak91} in 1982, also provides a way to look elements of various algebras as monotone ordered pair of sets. 
There are various algebraic representations in terms of rough sets. We mention here some well known representation results. For a good expositions of various algebraic representation results  in terms of rough sets, we refer to \cite{BC04}.
\begin{enumerate} \item (Comer \cite{Comer95}) Every regular double Stone algebra is isomorphic to an algebra of rough sets in a Pawlak approximation space. 
\item (Pagliani, \cite{Pa96}) Any finite semi simple Nelson algebra is isomorphic to a Nelson algebra formed by rough sets for some appropriate approximation space.
\item  (J\"{a}rvinen, Radeleczki \cite{JR11}) 
Every Nelson algebra defined over an algebraic lattice is isomorphic to an algebra of rough sets in an approximation space based on a quasi order. 
\item  (Kumar, Banerjee \cite{KB17}) 
Every Kleene algebra  is isomorphic to an algebra of rough sets in in a Pawlak approximation space. 
\end{enumerate}
\noindent Since Pawlak introduced the rough set theory, the algebraic and logical study of rough set theory evolved simultaneously. In fact, there are some works where the appearance of logics are motivated by the algebraic representations of rough sets. It is  noteworthy to mention here some logics which arose in context of rough set theory.
\begin{enumerate}
\item (Banerjee and Chakraborty \cite{BC96}) The emergence of Pre-rough logic was motivated by algebraic  representation of Pre-rough algebra in terms of rough sets. It has been shown (in \cite{BC96})
that Pre-rough logic is sound and complete in class of all Pre-rough algebra formed by rough sets, for all approximation spaces. 
\item (J\"{a}rvinen, Radeleczki \cite{JR11}, J\"{a}rvinen, Pagliani and Radeleczki \cite{JPR12}) Constructive logic with strong negation (CLSN) is sound and complete in class of all finite Nelson algebras formed by rough sets, for all  approximation spaces.
\item (Kumar and Banerjee \cite{KB17}) The $\mathcal{L}_{K}$ emerged as a result of rough set representation of Kleene algebras. Moreover, the logic $\mathcal{L}_{K}$ is sound and complete in class of all Kleene algebras formed rough sets, for all  approximation spaces.    
\end{enumerate}

\noindent In this work also, our  algebraic and logical study of rough set theory evolved simultaneously.  Rough set representations of Stone and dual Stone algebra which further leads to equivalency between rough set semantics, 3-valued semantics and algebraic semantics of the logics $\mathcal{L}_{\mathcal{S}}$ and $\mathcal{L}_{\mathcal{DS}}$.

The study of logics in this article is  completely based on  $DLL$. Let us present the logic.
The language consists of
\begin{itemize} 
\item Propositional variables:  $p,q,r, \ldots$. 
\item Logical connectives: $\vee, \wedge$.
\end{itemize}
\noindent The well-formed formulas of the logic are then given by the scheme: \\\hspace*{2 cm} $p ~| ~\alpha \vee \beta~ |~ \alpha \wedge \beta~ $. \vskip 2pt
\begin{notation} Denote the set of  propositional variables by $\mathcal{P}$, and that of well-formed formulas by $\mathcal{F}$. 
\end{notation}
\noindent Let $\alpha$ and $\beta$ be two formulas. The pair $(\alpha,\beta)$ is called a {\it consequence pair}. The rules and postulates of the logic $DLL$ are presented in terms of consequence pairs. Intuitively, the consequence pair $(\alpha,\beta)$ reflects that $\beta$ is a consequence of $\alpha$. In the representation of a logic, a consequence pair $(\alpha,\beta)$ is denoted by $\alpha \vdash \beta$ (called a {\it consequent}). 
The logic 
is now given through the following postulates and rules, taken from Dunn's \cite{Dunn99} and \cite{Dunn05}. These define reflexivity and transitivity of $\vdash$, introduction, elimination principles and the distributive law for the connectives $\wedge$ and $\vee$. 

\begin{definition} 
 {\rm ($DLL$- postulates)
\begin{enumerate}
\item $\alpha \vdash \alpha$ \hspace{1 cm}(Reflexivity). 
\item $\alpha \vdash \beta , \beta \vdash \gamma / \alpha \vdash \gamma$ \hspace{1 cm}(Transitivity).
\item $\alpha \wedge \beta \vdash \alpha$, $\alpha \wedge \beta \vdash \beta$ \hspace{1 cm}(Conjunction Elimination)
\item $\alpha \vdash \beta , \alpha \vdash \gamma / \alpha \vdash \beta \wedge \gamma$ \hspace{1 cm} (Conjunction Introduction)
\item $\alpha \vdash \alpha \vee \beta $, $\beta \vdash \alpha \vee \beta$ \hspace{1 cm}(Disjunction Introduction)
\item $\alpha \vdash \gamma , \beta \vdash \gamma/\alpha \vee \beta \vdash \gamma$ \hspace{1 cm} (Disjunction Elimination)
\item $\alpha \wedge (\beta \vee \gamma) \vdash (\alpha \wedge \beta) \vee (\alpha \wedge \gamma)$ \hspace{1 cm}(Distributivity)
\end{enumerate}}
\end{definition}
Further Dunn in \cite{Dunn95} extended the language of $DLL$ by adding,
\begin{itemize}
\item Propositional constants: $\top, \bot$.
\end{itemize}
Then, he added the following postulate to extend $DLL$ to give a logic $BDLL$, whose algebraic models are bounded distributive lattices.
\begin{itemize}
\item  $\alpha \vdash \top$ (Top);
$\bot \vdash \alpha$ (Bottom).
\end{itemize}
In this paper, semantics of a logic is  defined via valuations. 
 Let $\mathcal{A} = (A,\vee,\wedge,f_{1},f_{2},0,1)$ be a lattice based algebra, and $\mathcal{F}_{f_{1},f_{2}}$ be extension of $\mathcal{F}$ by adding unary connectives $f_{1}$ and $f_{2}$. 
A map $v:\mathcal{F}_{f_{1},f_{2}}  \rightarrow A$ is called a valuation on $A$ if
\begin{enumerate}
\item $v(\alpha \wedge \beta) = v(\alpha) \wedge \vee(\beta)$, 
\item $v(\alpha \vee \beta) = v(\alpha) \vee \vee(\beta)$.
\item $v(f_{1} \alpha) = f_{1} v(\alpha)$, $v(f_{2} \alpha) = f_{2} v(\alpha)$.
\item $v(\bot) = 0 $, $v(\top) = 1$. 
\end{enumerate} 
A consequent $\alpha \vdash \beta$ is {\it valid in A under the valuation} $v$, if $v(\alpha) \leq v(\beta)$. If the consequent is valid under all valuations on $A$, then it  is {\it valid in A}, and denote it as $\alpha \vDash_{A} \beta$.  Let $\mathcal{A}$ be a class of  algebras. If the consequent $\alpha \vdash \beta$ is valid in each algebra of $\mathcal{A}$, then we say $\alpha \vdash \beta$ {\it is valid in} $\mathcal{A}$, and denote it as $\alpha \vDash_{\mathcal{A}} \beta$.

This paper is organized as follows. In Section \ref{sec2}, we provide structural representations of Stone and dual Stone algebras in which elements of these algebra are defined by pairs of Boolean elements and rough sets. Negations in corresponding algebras are defined using Boolean complements. Moreover, we also obtained the 3-valued and rough set semantics of the proposed logics $\mathcal{L}_{S}$ and $\mathcal{L}_{DS}$. In Section \ref{sec3}, we show that each element of a double Stone algebra can be represented as ordered monotone 3-tuple of Boolean elements, and hence  3-tuple of sets. This leads to a 4-valued semantics of the proposed logic $\mathcal{L}_{DBS}$. Finally, we conclude this article in Section \ref{sec4}.

\vskip 2pt The lattice theoretic results used in this article are taken from \cite{davey}. We  use the convention of representing a set $\{x,y,z,...\}$ by $xyz....$.

\section{Boolean representations of Stone and dual Stone  algebras and corresponding logics} \label{sec2}
It has been a general trend in algebra to construct a new \textbf{type} of algebra from a given class of algebras. Some well known examples of such constructions are:
\bit
\item  Nelson algebra from a given Heyting algebra (Vakarelov \cite{Vakarelov77}, Fidel \cite{Fidel78}).
\item Kleene algebras from distributive lattices (Kalman \cite{KAL58}).
\item 3-valued {\L}ukasiewicz-Moisil (LM) algebra from a given Boolean algebra (Moisil, cf. \cite{Cignoli07}).
\item Regular double Stone algebra from a Boolean algebra  (Katri\v{n}\'{a}k \cite{KATRINAK74}, cf. \cite{Boicescu91}).
\eit
\noindent  More importantly, the afore mentioned constructions can be reversed in the sense of representations  of these new \textbf{type} of algebras in terms of the given class of algebras. In the same lines, our work in this section, is based on  Moisil's construction of a 3-valued LM algebra.  
As we mentioned in the Introduction, De Morgan, Kleene, Stone and double Stone algebras appear as reduct of  3-valued LM algebra (c.f. \cite{Boicescu91}). Exploiting this fact, we prove structural representation theorems for Stone and dual Stone algebras and provide 3-valued semantics for logics of  Stone and dual Stone algebras.
 
 
\subsection{Boolean representations of Stone and  dual Stone  algebras}
Let $\mathcal{B} := (B,\vee,\wedge,^{c},0,1)$ be a Boolean algebra, the set $B^{[2]} := \{(a,b): a\leq b, a,b \in B\}$ was first studied by Moisil, while constructing the algebraic models for 3-valued  {\L}ukasiewicz logic. In fact, Moisil showed that the structure $(B^{[2]},\vee,\wedge,^{\prime},\Delta,(0,0),(1,1))$ is a 3-valued {\L}ukasiewicz-Moisil (LM) algebra, where $\vee$, $\wedge$ are component wise operations, and for $(a,b) \in B^{[2]}$ $(a,b)^{\prime} = (b^{c},a^{c})$ and  $\Delta(a,b) = (a,a)$. Moreover, the following structural representation result was proved by Moisil.
\begin{theorem}(cf. \cite{Cignoli07})
Let $\textbf{LM} = (LM, \vee,\wedge,^{\prime},\Delta,0,1)$ be a 3-valued {\L}ukasiewicz-Moisil (LM) algebra, then  there exists a Boolean algebra $B$ such that $\textbf{LM}$ is embeddable  into  $(B^{[2]},\vee,\wedge,^{\prime},\Delta,(0,0),(1,1))$.

\end{theorem}
\noindent It is well known that 3-valued {\L}ukasiewicz-Moisil  algebra, regular double Stone algebra and semi simple Nelson algebra are equivalent in the sense that one can be obtained from the other by providing appropriate transformations. Hence a 3-valued {\L}ukasiewicz-Moisil (LM) algebra is also an Stone and a double Stone algebra.  
%
\begin{proposition}   \cite{Boicescu91}
\begin{enumerate}
\item $\mathcal{B}_{\sim}^{[2]} := (B^{[2]}, \vee, \wedge, \sim, (0,0), (1,1))$ is a Stone algebra, where, for $(a,b) \in B^{[2]}$, $\sim(a,b) := (b^{c},b^{c})$.
\item $\mathcal{B}_{\neg}^{[2]} := (B^{[2]}, \vee, \wedge, \neg, (0,0), (1,1))$ is a dual Stone algebra, where, for $(a,b) \in B^{[2]}$, $\neg(a,b) := (a^{c},a^{c})$.
\end{enumerate}
\end{proposition}
Let us  demonstrate the proof of $1$, proof of $2$ follows similarly.
\begin{proof}  
\noindent $(a,b) ~\wedge  (c,d) = (a\wedge c, b\wedge d) = (0,0)$. So, $c \leq a^{c}$ and $d \leq b^{c}$. Hence $(c,d) \leq (b^{c},b^{c})$. Clearly, we have $(a,b) ~\wedge  (b^{c},b^{c}) =  (0,0)$.
 Hence $\sim(a,b) = (b^{c},b^{c})$.\qed

\end{proof}
It is well known that, with pseudo negation $\sim$ and dual pseudo negation $\neg$, $\textbf{1}$, $\textbf{2}$ and $\textbf{3}$ are the only subdirect irreducible Stone and dual Stone algebras 
(Fig. \ref{fig1}). 
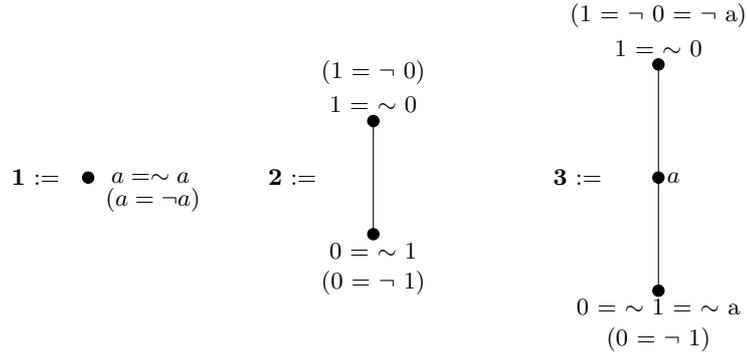
\begin{figure}[h] 
\begin{tikzpicture}[scale=.75]
    
    \draw [dotted] (0,1) -- (0,1);

		\draw [fill] (3,2) circle [radius = 0.1];
		\node [right] at (3.25,2) {$a = \sim a$};
		\node [right] at (3.15,1.60) {($a = \neg a$)};
		\node [right] at (1.5,2) {$\textbf{1}$ := };

		\node [right] at (6,2) {$\textbf{2}$ := };
		\draw (8,1) -- (8,3);
		\draw [fill] (8,1) circle [radius = 0.1];
		\node [below] at (8,1) {0 = $\sim$ 1};
		\node [below] at (8,0.5) {(0 = $\neg$ 1)};
		\draw [fill] (8,3) circle [radius = 0.1];
		\node [above] at (8,3) {1 = $\sim$ 0};
		\node [above] at (8,3.5) {(1 = $\neg$ 0)};
		
		\node [right] at (11,2) {$\textbf{3}$ := };
		\draw (13,0) -- (13,2);
		\node [below] at (13,0) {0 = $\sim$ 1 = $\sim$ a};
		\node [below] at (13,-0.5) {(0 = $\neg$ 1)};
    \draw [fill] (13,0) circle [radius = 0.1];
		\node [right] at (13,2) {$a $};
    \draw [fill] (13,2) circle [radius = 0.1];
		\draw (13,2) -- (13,4);
		\draw [fill] (13,4) circle [radius = 0.1];
		\node [above] at (13,4) {1 = $\sim$ 0};
		\node [above] at (13,4.5) {(1 = $\neg$ 0 = $\neg$ a)};

\end{tikzpicture}
\caption{Subdirectly irreducible Stone ($\sim$) and dual Stone ($\neg$) algebras}

\label{fig1}
\end{figure}
\noindent Hence using Birkhoff \cite{BIRK44} well known representation result, we have the followings. To distinguish the Stone and dual Stone algebra based on lattice $\textbf{3}$, we use $\textbf{3}_{\sim}$ and $\textbf{3}_{\neg}$  respectively. 
\begin{theorem}  \cite{Balbes74}
\begin{enumerate}
\item Let $\mathcal{S} = (S,\vee,\wedge,\sim,0,1)$ be an Stone algebra. There exists a (index) set $I$ such that $\mathcal{S}$ can be 
embedded into Stone algebra {\rm$\textbf{3}_{\sim}^{I}$}.
\item Let $\mathcal{DS} = (DS,\vee,\wedge,\neg,0,1)$ be a dual Stone algebra. There exists a (index) set $I$ such that $\mathcal{S}$ can be 
embedded into dual Stone algebra {\rm$\textbf{3}_{\neg}^{I}$}.

\end{enumerate}
\end{theorem}
\noindent So, if $B$ is a Boolean algebra, then the Stone algebra $B^{[2]}_{\sim}$ and dual Stone algebra $B^{[2]}_{\neg}$ are embeddable into $\textbf{3}_{\sim}^{I}$ and 
$\textbf{3}_{\neg}^{J}$ respectively, for appropriate index sets $\textbf{I}$ and $\textbf{J}$. 

 Atoms play fundamental role in the study of Boolean algebras. Completely join irreducible elements of lattices are counterpart of atoms in Boolean algebras, and play fundamental role in establishing isomorphism between  certain classes of lattice based algebra. An example of such can be seen in rough set representation \cite{JR11} of Nelson algebras. 
\begin{definition}\cite{davey}\label{cjijd} Let  $\mathcal{L}: =
(L,\vee,\wedge,0,1)$ be a complete lattice. \blr \item An element $a
\in L$ is said to be {\rm completely join irreducible}, if $a = \bigvee
S$ implies that $a \in S$, for every subset $S$ of $L$. \vskip 2pt 
 \begin{notation}\label{not1} Let
$\mathcal{J}_{L}$ denote the set of all   completely
join irreducible elements of $L$, and  $J(x) := \{a \in \mathcal{J}_{L}: a \leq x\}$, for any $x \in L$. \end{notation} 
\vskip 2pt 
\item A set $S$ is said to be {\rm join
dense} in $\mc L$, provided for every element $a \in L$, there is a subset $S^{\prime}$ of $S$ such that $a = \bigvee S^{\prime}$.
\elr
\end{definition}
The illustration of importance of completely join irreducible elements can be seen by a result of Birkhoff.
\begin{lemma} \cite{BIRK95} \label{lemma1}
Let $L$ and $K$ be two completely distributive lattices. Further, let $\mathcal{J}_{L}$ and $\mathcal{J}_{K}$ be join dense in $L$ and
$K$ respectively. Let $\phi:\mathcal{J}_{L} \rightarrow \mathcal{J}_{K}$ be an order isomorphism. Then the extension map $\Phi: L \rightarrow K$ given by \\
\hspace*{2 cm} $\Phi(x) := \bigvee(\phi(J(x)))$ (where $J(x) := \{a \in \mathcal{J}_{L}: a \leq x\}$), $x \in L$,\\
is a lattice isomorphism.
\end{lemma}
\noindent In \cite{KB17} we characterized the completely join irreducible elements of lattices $\textbf{3}^{I}$ and $B^{[2]}$ ,where $B$ is a complete atomic Boolean algebra.
 Let $i,k \in I$. Denote by  $f_{i}^{x},~x \in \{a,1\},$ 
the following element in $\textbf{3}^{I}$. \vskip 2pt
\hspace*{3 cm}$f_{i}^{x}(k)  :=
\left\{
	\begin{array}{ll}
		x  & \mbox{if } k = i \\
		0 & otherwise
	\end{array}
\right.$  
\begin{proposition} \label{pp2} \cite{KB17}
\begin{enumerate}
\item The set of completely join irreducible  elements of {\rm $\textbf{3}^{I}$}  is given by:\\
\hspace*{2 cm}{\rm$\mathcal{J}_{\textbf{3}^{I}} = \{f_{i}^{a}, f_{i}^{1}: i \in I \}$}.\vskip 2pt
\noindent Moreover, {\rm$\mathcal{J}_{\textbf{3}^{I}}$} 
is join dense in {\rm$\textbf{3}^{I}$}.
\item Let $B$ be a complete atomic Boolean algebra. The set of completely join irreducible  elements  of $B^{[2]}$ is given by\\
\hspace*{2 cm} $\mathcal{J}_{B^{[2]}} = \{(0,a),(a,a): a \in \mathcal{J}_{B}\}$.\vskip 2pt 
\noindent Moreover, $\mathcal{J}_{B^{[2]}}$ is join dense in $B^{[2]}$.
\end{enumerate}
\end{proposition}
\noindent  Figure \ref{fig2} shows the Hasse diagrams of {\rm$\mathcal{J}_{\textbf{3}^{I}}$} and $\mathcal{J}_{B^{[2]}}$.
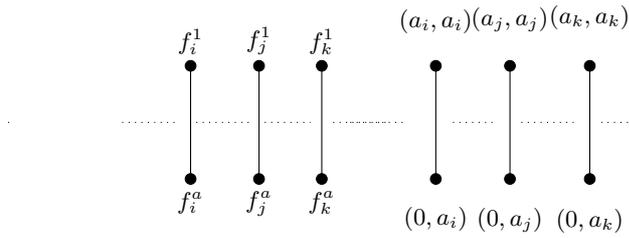
\begin{figure}[h] 
\begin{tikzpicture}[scale=.75]
    \tikzstyle{every node}=[draw,circle,fill=black,minimum size=4pt,
                            inner sep=0pt]

    \draw [dotted] (0,1) -- (0,1);
    \draw [dotted] (2,1) -- (3,1);
    
    \draw (3.2,0) -- (3.2,2);
		\draw (3.2,0) node [label=below:{\it $f_{i}^{a}$}]{};
		\draw (3.2,2) node [label=above:{\it $f_{i}^{1}$}]{};
		\draw [dotted] (3.3,1) -- (4.3,1);
		\draw (4.4,0) -- (4.4,2);
		\draw (4.4,0) node [label=below:{\it $f_{j}^{a}$}]{};
		\draw (4.4,2) node [label=above:{\it $f_{j}^{1}$}]{};
		\draw [dotted] (4.5,1) -- (5.3,1);
		\draw (5.5,0) -- (5.5,2);
		\draw (5.5,0) node [label=below:{\it $f_{k}^{a}$}]{};
		\draw (5.5,2) node [label=above:{\it $f_{k}^{1}$}]{};
		\draw [dotted] (5.7,1) -- (6.7,1);

		\draw [dotted] (6,1) -- (7,1);
  
    \draw (7.5,0) -- (7.5,2);
		\draw (7.5,0) node [label=below:{\it $(0,a_{i})$}]{};
		\draw (7.5,2) node [label=above:{\it $(a_{i},a_{i})$}]{};
		\draw [dotted] (7.8,1) -- (8.5,1);
		\draw (8.8,0) -- (8.8,2);
		\draw (8.8,0) node [label=below:{\it $(0,a_{j})$}]{};
		\draw (8.8,2) node [label=above:{\it $(a_{j},a_{j})$}]{};
		\draw [dotted] (9,1) -- (10,1);
		\draw (10.2,0) -- (10.2,2);
		\draw (10.2,0) node [label=below:{\it $(0,a_{k})$}]{};
		\draw (10.2,2) node [label=above:{\it $(a_{k},a_{k})$}]{};
		\draw [dotted] (10.4,1) -- (11,1);
  
\end{tikzpicture}
\caption{Hasse diagram of $\mathcal{J}_{\textbf{3}^{I}}$}
\label{fig2}
\end{figure} 
\noindent We also established the following isomorphism. 
\begin{theorem}\cite{KB17}
 The  sets of completely join irreducible elements of 
 {\rm $\textbf{3}^{I}$} and {\rm $(\textbf{2}^{I})^{[2]}$} are order isomorphic.
\end{theorem}
Now, we know that the pseudo and dual pseudo negations (if exist) are defined via the order of the given partially ordered sets. Moreover  Stone and dual Stone algebras are equational algebras, hence using Lemma \ref{lemma1}, we can deduce the following Theorem.  
\begin{theorem}\label{thm4}
\begin{enumerate}
 \item The  algebras {\rm $\textbf{3}^{I}$} and {\rm $(\textbf{2}^{I})^{[2]}$}
are lattice isomorphic.
\item The Stone (dual Stone) algebras {\rm $\textbf{3}^{I}$} and {\rm $(\textbf{2}^{I})^{[2]}$}
are  isomorphic.
\item Let $\mathcal{S} $ be an Stone  algebra. There exists a (index) set $I$ such that $\mathcal{S}$ can be 
embedded into Stone algebra {\rm $(\textbf{2}^{I})^{[2]}$}.

\item Let $\mathcal{DS}$ be a dual Stone algebra. There exists a (index) set $I$ such that $\mathcal{DS}$ can be 
embedded into dual Stone algebra {\rm $(\textbf{2}^{I})^{[2]}$}.
\end{enumerate}
\end{theorem}
\subsection{3-valued semantics of logic for Stone and  dual Stone  algebras}
As mentioned earlier, Moisil in 1941 (cf. \cite{Cignoli07}) proved that $B^{[2]}$ forms a 3-valued LM algebra. So, while discussing the  logic corresponding to the structures $B^{[2]}$, one is naturally led to  3-valued {\L}ukasiewicz logic. Our focus in this section is to study the logic corresponding to the classes  of Stone and dual Stone algebras and the structures $B^{[2]}_{\sim}$ and $B^{[2]}_{\neg}$. 
 Our approach to the study is 
motivated by Dunn's 4-valued semantics of the De Morgan consequence system \cite{Dunn99}: $\vDash_{0,1}$ (or $\vDash_{0}$ or $\vDash_{1}$), wherein valuations are defined in the De Morgan algebra $\textbf{4}$ (Figure \ref{fig7}).
\begin{figure}[h]
\begin{tikzpicture}[scale=.75]
                            inner sep=0pt]
    \draw [dotted] (0,1) -- (0,1);

		\draw (7,4) -- (5,2);
		\draw [fill] (5,2) circle [radius = 0.1];
		\node [left] at (4.75,2) {$n = \sim n$};
		\draw (7,0) -- (5,2);
		\draw (7,4) -- (9,2);
		\draw [fill] (9,2) circle [radius = 0.1];
		\node [right] at (9.25,2) {$b= \sim b$};
		\draw (7,0) -- (9,2);
		\draw [fill] (7,0) circle [radius = 0.1];
		\node [below] at (7,-0.25) {$f = \sim t$};
		
		\draw [fill] (7,4) circle [radius = 0.1];
		\node [above] at (7,4.25) {$t = \sim f$};

\end{tikzpicture}
\caption{De Morgan lattice $\textbf{4}$}
\label{fig7}
\end{figure}
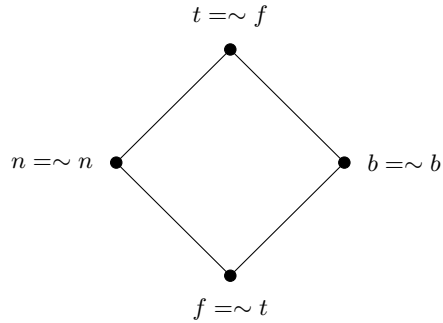

\noindent The 4-valued semantics arises from the fact that each element of a De Morgan algebra can be looked upon as a pair of sets. 
In this connection, we exploit Theorem \ref{thm4} to provide a 3-valued semantics of logic for Stone and  dual Stone  algebras. 
However, an easy consequence of   
Stone's representation theorem and Theorem \ref{thm4}, we have:
\begin{theorem}\label{cor1}
\begin{enumerate}
\item Given an Stone algebra $\mathcal{S} = (S,\vee,\wedge,\sim,0,1)$, there exist a set $U$ such that $\mathcal{S}$ can be embeddable into Stone algebra formed by 
$(\mathcal{P}(U))^{[2]}$. 
\item Given an dual Stone algebra $\mathcal{DS} = (DS,\vee,\wedge,\sim,0,1)$, there exist a set $U$ such that $\mathcal{DS}$ can be embeddable into dual Stone algebra formed by 
$(\mathcal{P}(U))^{[2]}$. 
\end{enumerate}
\end{theorem}
The 2-valued property of classical set theory arises from the fact that give a set $U$  and $A \subseteq U$, then  $\{A, A^{c}\}$ is a partition of the set $U$. So, for $x \in U$, either $x \in A$ or $x \in A^{c}$. Hence if $v$ is a valuation from classical propositional sentences to $\mathcal{P}(U)$ for some set $U$, then $v$ determine a class of 2-valued valuations $\{v_{x}: x \in U\}$  on classical propositional sentences, where where $v_{x}(\gamma) = 1$ if $x \in v(\gamma)$ and $v_{x}(\gamma) = 0$   if $x \notin v(\gamma)$. This shows the 2-valued semantics and set theoretic semantics of classical propositional are equivalent.   Now, the Stone's representation theorem for Boolean algebras guarantee that algebraic (Boolean algebras) semantics, 2 -valued semantics and set theoretic semantics of propositional logic are equivalent. In this section, we follow the same approach to establish the completeness results for $\mathcal{L}_{\mathcal{S}}$, $\mathcal{L}_{\mathcal{DS}}$  and $\mathcal{L}_{\mathcal{DBS}}$ (defined below).   




Let $\mathcal{F}_{\sim}$ and $\mathcal{F}_{\neg}$ be extensions of $\mathcal{F}$ by adding unary connectives $\sim$ and $\neg$ respectively.
 Let $\alpha,\beta \in \mathcal{F}_{\sim}$ and $\mathcal{L}_{S}$  denote the logic $BDLL$ along with following rules and postulates.
\begin{enumerate}
\item $\alpha \vdash \beta / \sim\beta \vdash  \sim\alpha $ \hspace{1 cm}(Contraposition) 
\item  $\sim \alpha \wedge \sim \beta \vdash \sim(\alpha \vee \beta)$ ($\vee$-linearity).
\item $\top \vdash \sim \bot$ (Nor).
\item  $\alpha \wedge \beta \vdash \gamma / \alpha \wedge \sim \gamma \vdash \sim \beta$,
\item $\alpha \wedge \sim\alpha \vdash \bot$,
\item $\top \vdash \sim \alpha \vee \sim \sim \alpha$,
\end{enumerate}
\noindent For $\alpha,\beta \in \mathcal{F}_{\neg}$, let $\mathcal{L}_{DS}$ denote the logic $BDLL$ along with following rules and postulates.
\begin{enumerate}
\item $\alpha \vdash \beta / \neg\beta \vdash  \neg\alpha$ \hspace{1 cm}(Contraposition) 
\item  $\neg(\alpha \wedge \beta)\vdash \neg \alpha \vee \neg \beta $ \hspace{1 cm} ($\wedge$-linearity).
\item $\neg \top \vdash \bot$
\item $\gamma \vdash \alpha \vee \beta / \neg \beta \vdash \alpha \vee \neg \gamma$.
\item $\top \vdash \alpha \vee \neg\alpha  $.
\item $\neg\alpha \wedge \neg\neg\alpha \vdash \bot$.
\end{enumerate}

\noindent Let $\mathcal{A}_{S}$ denote the class of all Stone algebras, $\mathcal{S}B^{[2]}$ denote the class of all Stone algebras formed by the set $B^{[2]}$ for all Boolean algebras $B$, $\mathcal{S}(\mathcal{P}(U)^{[2]}$ denote the class of all Stone algebras formed by the collection $\mathcal{P}(U)^{[2]}$ for all sets $U$. Now, utilizing Theorem \ref{thm4} and Corollary \ref{cor1}, in classical manner we have. 
\begin{theorem}\label{thm6} For $\alpha,\beta \in \mathcal{F}_{\sim}$, the followings are equivalent:
\begin{enumerate}
\item $\alpha \vdash_{\mathcal{L}_{S}} \beta$.  
\item $\alpha \vDash_{\mathcal{A}_{S}} \beta$.
\item $\alpha \vDash_{\mathcal{S}B^{[2]}} \beta$.
\item $\alpha \vDash_{\mathcal{S}(\mathcal{P}(U))^{[2]}} \beta$.
\end{enumerate}

\end{theorem}
In similar manner utilizing Theorem \ref{thm4} and Theorem \ref{cor1}, we have results for the logic $\mathcal{L}_{DS}$ and dual Stone algebras.   Let $\mathcal{A}_{DS}$ denote the class of all dual Stone algebras, $\mathcal{DS}B^{[2]}$ denote the class of all dual Stone algebras formed by the set $B^{[2]}$ for all Boolean algebras $B$, $\mathcal{DS}(\mathcal{P}(U)^{[2]}$ denote the class of all dual Stone algebras formed by the collection $\mathcal{P}(U)^{[2]}$ for all sets $U$. 
\begin{theorem}\label{thm7} For $\alpha,\beta \in \mathcal{F}_{\neg}$, the followings are equivalent:
\begin{enumerate}
\item $\alpha \vdash_{\mathcal{L}_{DS}} \beta$.  
\item $\alpha \vDash_{\mathcal{A}_{DS}} \beta$.
\item $\alpha \vDash_{\mathcal{DS}B^{[2]}} \beta$.
\item $\alpha \vDash_{\mathcal{DS}(\mathcal{P}(U))^{[2]}} \beta$.

\end{enumerate}
\end{theorem}
Now, let us define the following semantic consequence relations.
\begin{definition}\begin{enumerate}
\item Let $\alpha,\beta \in \mathcal{F}_{\sim}$.\vskip 2pt
\blr
\item   $\alpha \vDash^{S}_{1} \beta$ if and only  if, for all valuations $v$ in $\textbf{3}_{\sim}$ if  $v(\alpha) = 1$ then  $v(\beta) = 1$~~ {\rm (Truth preservation)}.
\item  $\alpha \vDash^{S}_{0} \beta$ if and only if, for all valuations $v$ in $\textbf{3}_{\sim}$ if $v(\beta) = 0$ then  $v(\alpha) = 0$  {\rm (Falsity preservation)}.
\item  $\alpha \vDash^{S}_{1,0} \beta$ if and only if,   $\alpha \vDash^{S}_{1} \beta$ {\it and} $\alpha \vDash^{S}_{0} \beta$.
\elr
\item Let $\alpha,\beta \in \mathcal{F}_{\neg}$.\vskip 2pt
\blr
\item   $\alpha \vDash^{DS}_{1} \beta$ if and only  if, for all valuations $v$ in $\textbf{3}_{\neg}$ if  $v(\alpha) = 1$ then  $v(\beta) = 1$~~ {\rm (Truth preservation)}.
\item   $\alpha \vDash^{DS}_{0} \beta$ if and only if, for all valuations $v$ in $\textbf{3}_{\neg}$ if $v(\beta) = 0$ then  $v(\alpha) = 0$  {\rm (Falsity preservation)}.
\item  $\alpha \vDash^{DS}_{1,0} \beta$ if and only if,  $\alpha \vDash^{DS}_{1} \beta$ {\it and} $\alpha \vDash^{DS}_{0} \beta$.
\elr
\end{enumerate}
\end{definition} 
\begin{proposition}\cite{Boicescu91} \label{pp3}
Let $\mathcal{S} = (S,\vee,\wedge,\sim,0,1)$ and $\mathcal{DS} = (DS,\vee,\wedge,\neg,0,1)$ be Stone and dual Stone algebra respectively, then for $a,b \in S$ and $x,y \in DS$
\blr
\item $\sim \sim( a \vee b) = \sim\sim a \vee \sim\sim b$ and $\sim \sim( a \wedge b) = \sim\sim a \wedge \sim\sim b$.
\item $\neg \neg( x \wedge y) = \neg\neg x \wedge \neg\neg y$ and $\neg \neg( x \vee y) = \neg\neg x \vee \sim\sim y$.
\elr
\end{proposition}
\begin{lemma}\label{lemma2}
\begin{enumerate}
\item For $\alpha,\beta \in \mathcal{F}_{\sim}$, if  $\alpha \vDash^{S}_{1} \beta$ then $\alpha \vDash^{S}_{0} \beta$.
\item For $\alpha,\beta \in \mathcal{F}_{\neg}$, if  $\alpha \vDash^{DS}_{0} \beta$ then $\alpha \vDash^{DS}_{1} \beta$.
\end{enumerate}
\end{lemma}
\begin{proof}
\begin{enumerate}
\item Let $\alpha \vDash^{S}_{1} \beta$, and $v$ be a valuation in  $\textbf{3}_{\sim}$ such that $v(\beta) = 0$. As $\alpha \vDash^{S}_{1} \beta$, so $v(\alpha) \neq 1$. If $v(\alpha) = 0$, then our work is done. So, assume that $v(\alpha) = a$. Define a map $v^{*}: \mathcal{F}_{\sim} \rightarrow \textbf{3}_{\sim} $ as:\\
\hspace*{3cm} $v^{*}(\gamma) = \sim\sim v(\gamma)$. 

\noindent Let us show that $v^{*}$ is indeed a valuation $\textbf{3}_{\sim}$. For this, we have to show that $v^{*}(\gamma \wedge \delta) = v^{*}(\gamma) \wedge v^{*}(\delta)$, $v^{*}(\gamma \vee \delta) = v^{*}(\gamma) \vee v^{*}(\delta)$, $v^{*}(\sim \gamma) = \sim v^{*}(\gamma)$, $v^{*}(\bot) = 0$ and $v^{*}(\top) = 1$, but this follows immediately from Proposition \ref{pp3}. 

Hence $v^{*}$ is a valuation and $v^{*}(\alpha) = 1$ and $v^{*}(\beta) = 0$ but this  contradicts to the fact that $\alpha \vDash^{S}_{1} \beta$. So, $\alpha \vDash^{S}_{1} \beta$ implies $\alpha \vDash^{S}_{0} \beta$.

\item Now, let $\alpha \vDash^{DS}_{0} \beta$, and $v$ be a valuation in  $\textbf{3}_{\neg}$ such that $v(\alpha) = 1$. As $\alpha \vDash^{DS}_{0} \beta$, so $v(\beta) \neq 0$. If $v(\beta) = 1$, then our work is done. So, assume that $v(\beta) = a$. In a similar fashion as in previous case, define a map $v^{*}: \mathcal{F}_{\neg} \rightarrow \textbf{3}_{\neg} $ as:\\
\hspace*{3cm} $v^{*}(\gamma) = \neg\neg v(\gamma)$. 

\noindent Similar to the previous case, using Proposition \ref{pp3}, we can easily established that $v^{*}$ is indeed a valuation $\textbf{3}_{\neg}$. This arises a contradiction to $\alpha \vDash^{DS}_{0} \beta$.


\end{enumerate}
\end{proof} \qed
\noindent Note that converse of the above statements are not true, for example 
$\sim\sim\alpha \vDash^{S}_{0} \alpha$ but $\sim\sim\alpha \nvDash^{S}_{1} \alpha$ and 
$\beta \nvDash^{DS}_{0} \neg\neg\beta$ but $\beta \vDash^{DS}_{1} \neg\neg\beta$. This is in contrary to the Dunn's consequence relations $\vDash_{0}$, $\vDash_{1}$ and $\vDash_{0,1}$ where all these three turn out be equivalent. 
\begin{theorem}
\begin{enumerate}
\item $\alpha \vDash_{\mathcal{SP}(U))^{[2]}} \beta$ if and only if $\alpha \vDash^{S}_{1} \beta$, for any $\alpha,\beta \in \mathcal{F}_{\sim}$.
\item $\alpha \vDash_{\mathcal{DSP}(U)^{[2]}} \beta$ if and only if $\alpha \vDash^{DS}_{0} \beta$, for any $\alpha,\beta \in \mathcal{F}_{\neg}$.
\end{enumerate}
\end{theorem}
\begin{proof}
\begin{enumerate}
\item Let $\alpha \vDash_{\mathcal{S}_{(\mathcal{P}(U))^{[2]}}} \beta$, and
 $v: \mathcal{F} \rightarrow \textbf{3}$ be a valuation. By Theorem \ref{cor1},
$\textbf{3}_{\sim}$ is embeddable to a Stone algebra of $\mathcal{P}(U)^{[2]}$ for some set $U$. If this embedding is denoted by $\phi$,  $\phi \circ v$ is  a valuation in  $\mathcal{P}(U)^{[2]}$. Then $(\phi \circ v)(\alpha) \leq (\phi \circ v)(\beta)$ implies $v(\alpha) \leq  v(\beta)$. Thus if  $v(\alpha)=1$, we have $v(\beta)=1$.\vskip 2pt

\noindent Now, let $\alpha \vDash^{S}_{1} \beta$. Let $U$ be a set, and $\mathcal{P}(U)^{[2]}$ be the corresponding Stone algebra.
  Let $v$ be a valuation on $\mathcal{P}(U)^{[2]}$ -- we need to show $v(\alpha) \leq v(\beta)$. For any $\gamma \in \mathcal{F}$ with $v(\gamma) := (A,B),~A,B \subseteq U,$ and for each $x \in U$, define a map $v_{x}: \mathcal{F} \rightarrow \textbf{3}_{\sim}$ as\\
\hspace*{2 cm}$v_{x}(\gamma)  :=
\left\{
	\begin{array}{lll}
		1 & \mbox{if }   x \in A \\
		a & \mbox{if } x \in B  \setminus A  \\
		0 & \mbox{if } x \notin B .
	\end{array}
\right.$ \vskip 2pt 
\noindent  Consider any $\gamma, \delta \in \mathcal{F}$, with  $v(\gamma) := (A,B)$ and $v(\delta) := (C,D),~A,B,C,D \subseteq U$. It is easy to show that (for a complete proof, we refer to  \cite{KB17}), $v_{x}(\gamma \wedge \delta) = v_{x}(\gamma) \wedge v_{x}(\delta)$, $v_{x}(\gamma \vee \delta) = v_{x}(\gamma) \vee v_{x}(\delta)$.
Let us show the following: $v_{x}(\sim\gamma) = \sim v_{x}(\gamma)$.

\noindent Note that $ v(\sim \gamma)= (B^{c},B^{c})$.\\
\noindent \un{Case 1} $v_{x}(\gamma) = 1$: Then $x \in A$, i.e. $x \notin A^{c}$ and so $x \notin B^{c}$. Hence 
          $v_{x}(\sim\gamma)  = 0 = \sim v_{x}(\gamma)$.\\
\noindent \un{Case 2} $v_{x}(\gamma) = a$:  $x \notin A$ but $x \in B$, so  $x \notin B^{c}$. Hence 
          $v_{x}(\sim\gamma)  = 0 = \sim v_{x}(\gamma)$.\\	
\noindent \un{Case 3} $v_{x}(\gamma) = 0$:  $x \notin B$, i.e. $x \in B^{c}$. So 
          $v_{x}(\sim\gamma)  = 1 = \sim v_{x}(\gamma)$.									


\noindent Hence $v_{x}$ is a valuation in $\textbf{3}_{\sim}$. Now let us show that $v(\alpha) \leq v(\beta)$. Let  $v(\alpha) := (A^{\prime},B^{\prime})$, $v(\beta) := (C^{\prime},D^{\prime})$, and $x \in A^{\prime}$. Then $v_{x}(\alpha) = 1$, and as $\alpha \vDash^{S}_{1} \beta$, by definition, $v_{x}(\beta) = 1$. This implies $x \in C^{\prime}$, whence $A^{\prime} \subseteq C^{\prime}$.\\
On the other hand, if $x \notin D^{\prime}$, $v_{x}(\beta) = 0$. Then using Lemma \ref{lemma2}, we have   $v_{x}(\alpha) = 0$, so that $x \notin B^{\prime}$, giving  $B^{\prime} \subseteq D^{\prime}$.

\item We prove second part only.   For this let $\alpha \vDash^{DS}_{0} \beta$. Let $U$ be a set, and $\mathcal{P}(U)^{[2]}$ be the corresponding dual Stone algebra.
  Let $v$ be a valuation on $\mathcal{P}(U)^{[2]}$,  we  show that $v(\alpha) \leq v(\beta)$. Very similar to previous case, for any $\gamma \in \mathcal{F}$ with $v(\gamma) := (A,B),~A,B \subseteq U,$ and for each $x \in U$, define a map $v_{x}: \mathcal{F} \rightarrow \textbf{3}_{\neg}$ as\\
\hspace*{2 cm}$v_{x}(\gamma)  :=
\left\{
	\begin{array}{lll}
		1 & \mbox{if }   x \in A \\
		a & \mbox{if } x \in B  \setminus A  \\
		0 & \mbox{if } x \notin B .
	\end{array}
\right.$ \vskip 2pt 
\noindent  Consider any $\gamma, \delta \in \mathcal{F}$, with  $v(\gamma) := (A,B)$ and $v(\delta) := (C,D),~A,B,C,D \subseteq U$. Let us show the following: $v_{x}(\neg\gamma) = \neg v_{x}(\gamma)$.

\noindent Note that $ v(\neg \gamma)= (A^{c},A^{c})$.\\
\noindent \un{Case 1} $v_{x}(\gamma) = 1$: Then $x \in A$, i.e. $x \notin A^{c}$. Hence 
          $v_{x}(\neg\gamma)  = 0 = \neg v_{x}(\gamma)$.\\
\noindent \un{Case 2} $v_{x}(\gamma) = a$:  $x \notin A$ but $x \in B$, so  $x \in A^{c}$. Hence 
          $v_{x}(\neg\gamma)  = 0 = \neg v_{x}(\gamma)$.\\	
\noindent \un{Case 3} $v_{x}(\gamma) = 0$:  $x \notin B$, and so $x \notin A$, i.e. $x \in A^{c}$. So 
          $v_{x}(\neg\gamma)  = 1 = \neg v_{x}(\gamma)$.									


\noindent Hence $v_{x}$ is a valuation in $\textbf{3}_{\neg}$. To complete the proof let us show that $v(\alpha) \leq v(\beta)$. Let  $v(\alpha) := (A^{\prime},B^{\prime})$, $v(\beta) := (C^{\prime},D^{\prime})$, and $x \in A^{\prime}$. Then $v_{x}(\alpha) = 1$, and as $\alpha \vDash^{DS}_{1} \beta$, by Lemma \ref{lemma2}, $v_{x}(\beta) = 1$. This implies $x \in C^{\prime}$, whence $A^{\prime} \subseteq C^{\prime}$.\\
On the other hand, if $x \notin D^{\prime}$, $v_{x}(\beta) = 0$. Then by our assumption $\alpha \vDash^{DS}_{0} \beta$, we have   $v_{x}(\alpha) = 0$, so that $x \notin B^{\prime}$, giving  $B^{\prime} \subseteq D^{\prime}$.
\end{enumerate}
\end{proof}\qed

\noindent Note that in proof of previous Theorem, while defining the valuations $v_{x}$ we have utilized the fact that if  $U$ is a set and $A,B \subseteq U$ with $A \subseteq B$, then the collection $\{A,B\setminus A, B^{c}\}$ is a partition of $U$. So, for $x \in U$, either $x \in A$ or $x \in B\setminus A$ or $x \in B^{c}$. Nevertheless, finally we have the following Theorem.
\begin{theorem}\label{thm9}(3-valued semantics)
\begin{enumerate}
\item $\alpha \vdash_{\mathcal{L}_{S}} \beta$ if and only if $\alpha \vDash^{S}_{t} \beta$.
\item $\alpha \vdash_{\mathcal{L}_{DS}} \beta$ if and only if $\alpha \vDash^{DS}_{f} \beta$.

\end{enumerate}
\end{theorem}
\subsection{Rough set models for  3-valued logics}
Rough set theory, introduced by Pawlak \cite{pawlak82} in 1982, deals with a domain $U$ that is the set of objects, and an equivalence ({\it indiscernibility}) relation $R$
on $U$. 
The pair $(U,R)$ is called an (Pawlak) {\it approximation space}.  
For any $A \subseteq U$, one defines the {\it lower} and {\it upper approximations} of $A$ in the approximation space $(U,R)$, denoted ${\sf L}A$ and ${\sf U}A$ respectively,  as follows.\\
\hspace*{2 cm} ${\sf L}A := \bigcup \{[x]: [x]\subseteq X\}$,\\
\hspace*{2 cm} ${\sf U}A := \bigcup \{[x]: [x] \cap X \neq \emptyset\}$. \hspace{1 cm} \hfill{(*)}\\
As the information about the objects of the domain is available modulo the equivalence classes in $U$, the description of any concept, represented extensionally as the subset $A$ of $U$, is inexact. One then  `approximates' the description from within and outside, through the lower and upper approximations respectively.
\noindent Unions of equivalence classes are termed as {\it definable} sets, signifying exactly describable concepts in the context of the given information. In particular, sets of the form 
${\sf L}A$, ${\sf U}A$ are definable sets. 
\begin{definition}
Let $(U,R)$ be an approximation space. For each $A \subseteq U$, the ordered pair $({\sf L}A,{\sf U}A)$ is called a {\rm rough set} in $(U,R)$.
\begin{notation}  $\mathcal{RS} := \{({\sf L}A,{\sf U}A): A \subseteq U\}$.
\end{notation}
The ordered pair $(D_{1},D_{2})$, where $D_{1} \subseteq D_{2}$ and $D_{1},D_{2}$ are definable sets, is called a {\rm generalized rough set} in $(U,R)$.
 \begin{notation}  $\mathcal{D} $ denotes the collection of definable sets and $\mathcal{R} $  that of the generalized rough sets in $(U,R)$.\end{notation}
\end{definition}
 In rough set theory, if $(U,R)$ is an approximation space and $A \subseteq U$ then  $\{{\sf L}A,{\sf U}A \setminus {\sf L}A, ({\sf U}A)^{c}\}$
is a partition of $U$. So for any $x \in U$, one of the following is true: either $x \in {\sf L}A$ or $x \in {\sf U}A \setminus {\sf L}A$ or $x \in ({\sf U}A)^{c}\}$. This fact leads to the following interpretations in rough set theory.
\begin{enumerate}
\item \label{t} $x$  {\it certainly} belongs to $A$, if $x \in {\sf L}A$,  i.e. all  objects which are indiscernible to $x$
are in $A$.

\item \label{f} $x$  {\it certainly does not} belong to $A$, if $x \notin {\sf U}A$,  i.e. all objects which are indiscernible to $x$
are not in  $A$.

\item \label{p} Belongingness of $x$ to $A$ is {\it not certain, but possible}, 
if $x \in {\sf U}A$ but $x \notin {\sf L}A$. In rough set terminology, this is the case when $x$ is in the {\it boundary} of $A$: some objects indiscernible
to $x$ are in $A$, while some others, also indiscernible to $x$, are in $A^{c}$.
\end{enumerate}
\noindent The phrase `certainly' in the above interpretations derived because of indiscernible behavior of the equivalence relation $R$. For any set $A \subseteq U$, ${\sf L}A, ({\sf U}A)^{c}$ are certain regions and the set  ${\sf U}A \setminus {\sf L}A$ is uncertain region.   

 Nevertheless, 
these interpretations have led to much work in the study of connections between 3-valued algebras or logics and rough sets, see for instance  \cite{Banerjee97,Dunt97,Pa98,Iturr99,Avron08,ciucci14}.
In \cite{KB17} author came up with a logic $\mathcal{L}_{K}$, which explicitly capture these interpretations via a 3-valued semantics. It also worth mention here the work of Avron and Konikowska, in \cite{Avron08}, they have 
 obtained a non-deterministic logical matrix and studied the 3-valued logic generated by this matrix. A simple predicate language is used, with no quantifiers or connectives, to express membership in rough sets. Connections, in special cases, with 3-valued Kleene, {\L}ukasiewicz and two paraconsistent logics are established. 

Now, let us provide rough set representations of Stone and dual Stone algebras and rough set semantics for $\mathcal{L}_{S}$ and $\mathcal{L}_{}DS$. Let  $\mathcal{RS}_{\sim}$ and $\mathcal{RS}_{\neg}$ denote respectively the Stone algebra  and dual Stone algebra formed by $\mathcal{RS}$ for an approximation space $(U,R)$.
\begin{theorem}\label{thm10}
\begin{enumerate}
\item \begin{enumerate}
\item Given an  Stone algebra $\mathcal{S} = (S,\vee,\wedge,\sim,0,1)$, then there exists an approximation space  $(U,R)$ such that $\mathcal{S}$ can be embedded into  Stone algebra  $\mathcal{RS}_{\sim}$.
\item Let $\mathcal{A_{SRS}}$ denote the class of all Stone algebras formed by $\mathcal{RS}$. Then we have for $\alpha ,\beta \in \mathcal{F}_{\sim}$:\\
$\alpha \vdash_{\mathcal{L}_{S}}$ if and only if $\alpha \vDash_{\mathcal{A_{SRS}}} \beta$. 
\end{enumerate}
\item 
\begin{enumerate} \item Given a dual Stone algebra $\mathcal{DS} = (DS,\vee,\wedge,\neg,0,1)$, then there exists an approximation space  $(U,R)$ such that $\mathcal{DS}$ can be embedded into dual Stone algebra $\mathcal{RS}_{\neg}$.
\item Let $\mathcal{A_{DSRS}}$ denote the class of all dual Stone algebras formed by $\mathcal{RS}$. Then we have for $\alpha ,\beta \in \mathcal{F}_{\neg}$:\\
$\alpha \vdash_{\mathcal{L}_{DS}}$ if and only if $\alpha \vDash_{\mathcal{A_{DSRS}}} \beta$. 
\end{enumerate}
\end{enumerate}
\end{theorem}
\begin{proof}
In \cite{KB17}, we proved that given a set $U^{\prime}$, there exists an approximation space $(U,R)$ such that the partially ordered set $(\mathcal{P}(U^{\prime})^{[2]},\leq)$ is order isomorphic to collection of rough sets $\mathcal{RS}$ corresponding to  $(U,R)$. This order isomorphism can easily be extended to an Stone or dual Stone isomorphism. This proves $1$(a) and $2$(a). $1$(b) and $2$(b) follows from $1$(a), $2$(a),
 Theorem \ref{thm6} and Theorem \ref{thm7}.  
\end{proof}\qed
\noindent Hence, in particular, Stone algebra $\textbf{3}_{\sim}$ and dual Stone algebra $\textbf{3}_{\neg}$ are isomorphic to $\mathcal{RS}$ for some approximation space. Figure \ref{fig3} and 
Figure \ref{fig4} depict the isomorphisms, where  $U = \{x\}$ and $\mathcal{RS}$ is collection of rough sets corresponding to approximation $(U^{\prime} = \{x,x^{\prime}\},R^{\prime} = \{(x,x),(x,x^{\prime}),(x^{\prime},x),(x^{\prime},x^{\prime})\})$.
\begin{figure}[h] 
\begin{tikzpicture}[scale=.65]
    
    \draw [dotted] (0,1) -- (0,1);
			\node [right] at (0,2) {$\textbf{3}_{\sim}$ := };
		\draw (2,0) -- (2,2);
		\node [below] at (2,0) {0 = $\sim$ 1 = $\sim$ a};
    \draw [fill] (2,0) circle [radius = 0.1];
		\node [right] at (2,2) {a };
    \draw [fill] (2,2) circle [radius = 0.1];
		\draw (2,2) -- (2,4);
		\draw [fill] (2,4) circle [radius = 0.1];
		\node [above] at (2,4) {1 = $\sim$ 0};
		\node [left] at (4,2) {$\cong$};
		\node [right] at (4.5,2) {$\mathcal{P}(U)^{2}_{\sim}: = $};
		\draw (7.5,0) -- (7.5,2);
		\node [below] at (7.5,0) {$(\emptyset,\emptyset) = \sim(\emptyset,U) = \sim(U,U)$};
    \draw [fill] (7.5,0) circle [radius = 0.1];
		\node [right] at (7.5,2) {$(\emptyset,U)$};
    \draw [fill] (7.5,2) circle [radius = 0.1];
		\draw (7.5,2) -- (7.5,4);
		\draw [fill] (7.5,4) circle [radius = 0.1];
		\node [above] at (7.5,4) {$(U,U) = \sim(\emptyset,\emptyset)$};
		\node [left] at (10,2) {$\cong$};
		\node [right] at (10.5,2) {$\mathcal{RS}_{\sim}: = $};
		\draw (13,0) -- (13,2);
		\node [below] at (13.5,0) {$({\sf  L}\emptyset,{\sf U}\emptyset) = \sim ({\sf L}U,{\sf L}U) $};
		\node [below] at (14,-0.5) {$ = \sim({\sf L }x,{\sf L}x) = \sim({\sf L }x^{\prime},{\sf L}x^{\prime})$};
    \draw [fill] (13,0) circle [radius = 0.1];
		\node [right] at (13,2) {$({\sf L }x,{\sf L}x) = ({\sf L }x^{\prime},{\sf L}x^{\prime})$};
    \draw [fill] (13,2) circle [radius = 0.1];
		\draw (13,2) -- (13,4);
		\draw [fill] (13,4) circle [radius = 0.1];
		\node [above] at (13,4) {$({\sf L}U,{\sf L}U) = \sim ({\sf  L}\emptyset,{\sf U}\emptyset)$};

\end{tikzpicture}
\caption{ $\textbf{3}_{\sim } \cong \mathcal{P}(U)^{2}_{\sim} \cong \mathcal{RS}_{\sim}$ }

\label{fig3}
\end{figure}
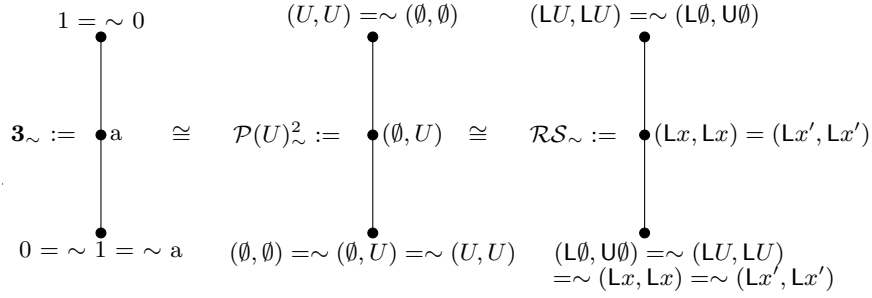 

\begin{figure}[ht] 
\begin{tikzpicture}[scale=.65]
    
    \draw [dotted] (0,1) -- (0,1);
			\node [right] at (0,2) {$\textbf{3}_{\neg}$ := };
		\draw (2,0) -- (2,2);
		\node [below] at (2,0) {0 = $\sim$ 1 = $\neg$ u};
    \draw [fill] (2,0) circle [radius = 0.1];
		\node [right] at (2,2) {u };
    \draw [fill] (2,2) circle [radius = 0.1];
		\draw (2,2) -- (2,4);
		\draw [fill] (2,4) circle [radius = 0.1];
		\node [above] at (2,4) {1 = $\neg$ 0};
		\node [left] at (4,2) {$\cong$};
		\node [right] at (4.5,2) {$\mathcal{P}(U)^{2}_{\neg}: = $};
		\draw (7.5,0) -- (7.5,2);
		\node [below] at (7.5,0) {$(\emptyset,\emptyset) = \neg(\emptyset,U) = \neg(U,U)$};
    \draw [fill] (7.5,0) circle [radius = 0.1];
		\node [right] at (7.5,2) {$(\emptyset,U)$};
    \draw [fill] (7.5,2) circle [radius = 0.1];
		\draw (7.5,2) -- (7.5,4);
		\draw [fill] (7.5,4) circle [radius = 0.1];
		\node [above] at (7.5,4) {$(U,U) = \neg(\emptyset,\emptyset)$};
		\node [left] at (10,2) {$\cong$};
		\node [right] at (10.5,2) {$\mathcal{RS}_{\neg}: = $};
		\draw (13,0) -- (13,2);
		\node [below] at (13.5,0) {$({\sf  L}\emptyset,{\sf U}\emptyset) = \neg ({\sf L}U,{\sf L}U) $};
    \draw [fill] (13,0) circle [radius = 0.1];
		\node [right] at (13,2) {$({\sf L }x,{\sf L}x) = ({\sf L }x^{\prime},{\sf L}x^{\prime})$};
    \draw [fill] (13,2) circle [radius = 0.1];
		\draw (13,2) -- (13,4);
		\draw [fill] (13,4) circle [radius = 0.1];
		\node [above] at (13,4) {$({\sf L}U,{\sf L}U) = \neg ({\sf  L}\emptyset,{\sf U}\emptyset)$};
		\node [above] at (13,4.5) {$ = \neg({\sf L }x,{\sf L}x) = \neg({\sf L }x^{\prime},{\sf L}x^{\prime})$};

\end{tikzpicture}
\caption{ $\textbf{3}_{\neg } \cong \mathcal{P}(U)^{2}_{\neg} \cong \mathcal{RS}_{\neg}$ }

\label{fig4}
\end{figure}
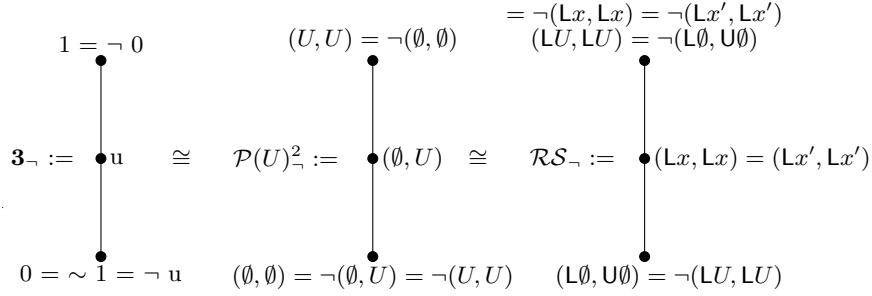 

\noindent Now, we would  explicate the relationship between rough sets and the 3-valued semantics of the logic  $\mathcal{L}_{S}$ and $\mathcal{L}_{DS}$  indicated in  Theorem \ref{thm10} and Theorem \ref{thm9}.
\noindent 
 
\noindent Let $\alpha$ be a formula in $\mathcal{L}_{S}$ and $\gamma$ be a formula in $\mathcal{L}_{DS}$. Let    $v$ be a valuation in $\mathcal{RS}_{\sim}$ for some approximation space $(U,R)$  such that $v(\alpha) := ({\sf L}A,{\sf U}A), A \subseteq U$, and $v^{\prime}$ be a valuation in  $\mathcal{RS}^{\prime}_{\neg}$ for some approximation space $(U^{\prime},R^{\prime})$ such that
  $v(\gamma) := ({\sf L}A^{\prime},{\sf U}A^{\prime}), A^{\prime} \subseteq U^{\prime}$. Let $x \in U$ and $y \in U^{\prime}$. 
We define the following semantic consequence relations. 
\begin{enumerate}
\item 
 $v,x\vDash^{\mathcal{RS}_{\sim}}_{1} \alpha$  if and only if $x \in {\sf L}A$.\\
 $v,x\vDash^{\mathcal{RS}_{\sim}}_{0} \alpha$  if and only if $x \notin {\sf U}A$.\\
 $v,x\vDash^{\mathcal{RS}_{\sim}}_{u} \alpha$  if and only if $x \notin {\sf L}A, x \in {\sf U}A$.
\item 
 $v^{\prime},y\vDash^{\mathcal{RS}_{\neg}}_{1} \gamma$  if and only if $y \in {\sf L}A^{\prime}$.\\
 $v^{\prime},y\vDash^{\mathcal{RS}_{\neg}}_{0} \gamma$  if and only if $y \notin {\sf U}A^{\prime}$.\\
 $v^{\prime},y\vDash^{\mathcal{RS}_{\neg}}_{u} \gamma$  if and only if $y \notin {\sf L}A, y \in {\sf U}A^{\prime}$.
\end{enumerate}
\noindent Then $v,x\vDash^{\mathcal{RS}_{\sim}}_{1} \alpha$,  $v^{\prime},y\vDash^{\mathcal{RS}_{\neg}}_{1} \alpha$ captures the interpretation \ref{t}, $v,x\vDash^{\mathcal{RS}_{\sim}}_{0} \alpha$, $v^{\prime},y\vDash^{\mathcal{RS}_{\neg}}_{0} \alpha$ captures the interpretation \ref{f} and $v,x\vDash^{\mathcal{RS}_{\sim}}_{u} \alpha$, $v^{\prime},y\vDash^{\mathcal{RS}_{\neg}}_{u} \alpha$ captures the interpretation \ref{p}. \\
Next, let us   define the following relations. Let $\alpha,\beta \in \mathcal{F}_{\sim}$ and $\gamma,\delta \in \mathcal{F}_{\neg}$, 
\begin{enumerate}
 \item $\alpha\vDash^{\mathcal{RS}_{\sim}}_{1}\beta$ if and only if $v,x\vDash^{\mathcal{RS}_{\sim}}_{1}\alpha$  implies $v,x\vDash^{\mathcal{RS}_{\sim}}_{1}\beta$, for all valuations $v$ in $\mathcal{RS}_{\sim}$ and $x \in U$.\\
$\alpha\vDash^{\mathcal{RS}_{\sim}}_{0}\beta$  if and only if $v,x\vDash^{\mathcal{RS}_{\sim}}_{0}\beta$ implies $v,x\vDash^{\mathcal{RS}_{\sim}}_{0}\alpha$, for all valuations $v$ in $\mathcal{RS}_{\sim}$ and $x \in U$.\\
 $\alpha\vDash^{\mathcal{RS}_{\sim}}_{1,0}\beta$  if and only if $\alpha\vDash^{\mathcal{RS}_{\sim}}_{1}\beta$ and $\alpha\vDash^{\mathcal{RS}_{\sim}}_{0}\beta$.
\item $\gamma\vDash^{\mathcal{RS}_{\neg}}_{1}\delta$ if and only if $v^{\prime},y\vDash^{\mathcal{RS}_{\neg}}_{1}\gamma$  implies $v^{\prime},y\vDash^{\mathcal{RS}_{\neg}}_{1}\delta$, for all valuations $v^{\prime}$ in $\mathcal{RS}_{\neg}$ and $y \in U^{\prime}$.\\
$\gamma\vDash^{\mathcal{RS}_{\neg}}_{0}\delta$  if and only if $v^{\prime},y\vDash^{\mathcal{RS}_{\neg}}_{0}\delta$ implies $v^{\prime},y\vDash^{\mathcal{RS}_{\neg}}_{0}\gamma$, for all valuations $v^{\prime}$ in $\mathcal{RS}_{\neg}$ and $y \in U^{\prime}$.\\
 $\gamma\vDash^{\mathcal{RS}_{\neg}}_{1,0}\delta$  if and only if $\gamma\vDash^{\mathcal{RS}_{\neg}}_{1}\delta$ and $\gamma\vDash^{\mathcal{RS}_{\neg}}_{0}\delta$.
\end{enumerate}

Now we link the syntax and semantics.
\begin{definition} \begin{enumerate}
\item Let $\alpha, \beta \in \mathcal{F}_{\sim}$ and $\alpha \vdash \beta$ be a consequent. 
\begin{itemize}
\item $\alpha \vdash \beta$ is valid in an approximation space  $(U,R)$, 
if and only if $\alpha\vDash^{\mathcal{RS}_{\sim}}_{1}\beta$.
\item $\alpha \vdash \beta$ is valid in a class $\mathcal{F}$  of approximation spaces if and only if $\alpha \vdash \beta$ is valid in all approximation spaces  $(U,R) \in \mathcal{F}$.
\end{itemize}

\item Let $\alpha, \beta \in \mathcal{F}_{\neg}$ and $\alpha \vdash \beta$ be a consequent. 
\begin{itemize}
\item $\alpha \vdash \beta$ is valid in an approximation space  $(U,R)$, 
if and only if $\alpha\vDash^{\mathcal{RS}_{\neg}}_{0}\beta$.
\item $\alpha \vdash \beta$ is valid in a class $\mathcal{F}$  of approximation spaces if and only if $\alpha \vdash \beta$ is valid in all approximation spaces  $(U,R) \in \mathcal{F}$.
\end{itemize}
\end{enumerate}
\end{definition}
\begin{theorem} Let Let $\alpha, \beta \in \mathcal{F}_{\sim}$ and  $\gamma, \delta \in \mathcal{F}_{\neg}$, then
\begin{enumerate}
\item $\alpha \vDash_{\mathcal{A}_{S\mathcal{RS}}} \beta$ if and only if $\alpha \vdash \beta$ is valid in the class of all approximation spaces.
\item $\gamma \vDash_{\mathcal{A}_{DS\mathcal{RS}}} \delta$ if and only if $\gamma \vdash \delta$ is valid in the class of all approximation spaces.
\end{enumerate}
\end{theorem}
\begin{proof}
\begin{enumerate}
\item Let $\alpha \vDash_{\mathcal{A}_{S\mathcal{RS}}} \beta$. Let $(U,R)$ be an approximation space, and $v$ be a valuation in $\mathcal{RS}_{\sim}$
with  $v(\alpha) := ({\sf L}A,{\sf U}A)$ and $v(\beta): = ({\sf L}B,{\sf U}B),~A,B \subseteq U$. 

\noindent 
By the assumption, ${\sf L}A \subseteq {\sf L}B$ and 
${\sf U}A \subseteq {\sf U}B$. Now, let us  show that $\alpha\vDash^{\mathcal{RS}_{\sim}}_{1}\beta$. So, let $x \in U$ and $v,x\vDash^{\mathcal{RS}_{\sim}}_{1} \alpha$, i.e., $x \in {\sf L}A$. But we have ${\sf L}A \subseteq {\sf L}B$, hence $v,x\vDash^{\mathcal{RS}_{\sim}}_{1} \beta$. 

\noindent Now, suppose $\alpha \vdash \beta$ is valid in the class of all approximation spaces. We want to show that $\alpha \vDash_{\mathcal{A}_{S\mathcal{RS}}} \beta$. Let $v$ be a valuation in $\mathcal{RS}_{\sim}$ as taken above. We have to show that ${\sf L}A \subseteq {\sf L}B$ and ${\sf U}A \subseteq {\sf U}B$. Let $x \in {\sf L}A$, i.e., $v,x\vDash^{\mathcal{RS}_{\sim}}_{1} \alpha$. Hence by our assumption, 
$v,x\vDash^{\mathcal{RS}}_{1} \beta$, i.e., $x \in {\sf L}B$. So  ${\sf L}A \subseteq {\sf L}B$. Now, let $y \notin {\sf U}B$, using Lemma \ref{lemma2}, we have 
$v,y \vDash^{\mathcal{RS}_{\sim}}_{0} \beta$. By our assumption, $v,y \vDash^{\mathcal{RS}_{\sim}}_{0} \alpha$, i.e., $y \notin {\sf U}A$.
\item The proof of this part is very similar to that of part 1 which uses lemma \ref{lemma2}.
\end{enumerate}
\end{proof}\qed



\section{Boolean representation of double Stone algebra and corresponding logic}\label{sec3}
In this section, we follow the same idea as in above section. We provide an structural representation of double Stone algebra in which negations present in double Stone algebra are described using Boolean complement.
\subsection{Boolean representation of double Stone algebra}
Moisil in 1941 (cf. \cite{Boicescu91}) generalizes the construction of $B^{[2]}$. He defined for any natural number $n$, the set 

\noindent \hspace*{2 cm} $B^{[n]} = \{(b_{1},b_{2},....,b_{n}) \in B^{n}: b_{1} \leq b_{2}\leq .... \leq b_{n}\}$,

\noindent as examples of n-valued {\L}ukasiewicz Moisil algebras. Our interest in this section is to consider $B^{[3]}$  as double Stone algebra. 
\begin{proposition}   \cite{Boicescu91}
$(B^{[3]}, \vee, \wedge, \sim, \neg, (0,0,0), (1,1,1))$ is a double Stone algebra, where, for $(a,b,c), (d,e,f) \in B^{[3]}$, \vskip 2pt 
\noindent 
 $(a,b,c) \vee (d,e,f) := (a\vee d,b \vee e, c \vee f )$,
 $(a,b,c) \wedge (d,e,f) := (a\wedge d,b \wedge e,c \wedge f )$,\\ 
\hspace*{2 cm} $\sim(a,b,c) := (c^{c},c^{c},c^{c})$, $\neg(a,b,c) := (a^{c},a^{c},a^{c})$, 
\end{proposition}
%
%
%
It is well known that $\textbf{1}$, $\textbf{2}$, $\textbf{3}$ and $\textbf{4}$ are the only sudirect irreducible (Figure \ref{fig5}) double Stone algebras. 
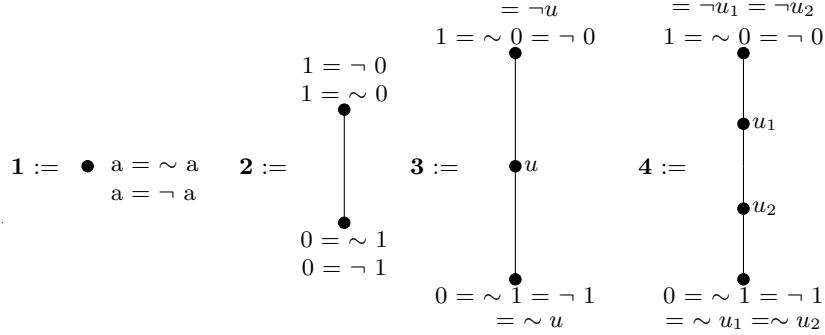
\begin{figure}[ht] 
\begin{tikzpicture}[scale=.75]
    
    \draw [dotted] (0,1) -- (0,1);

		\draw [fill] (1.5,2) circle [radius = 0.1];
		\node [right] at (1.75,2) {a = $\sim$ a};
		\node [right] at (1.75,1.5) {a = $\neg$ a};
		\node [right] at (0,2) {$\textbf{1}$ := };

		\node [right] at (4,2) {$\textbf{2}$ := };
		\draw (6,1) -- (6,3);
		\draw [fill] (6,1) circle [radius = 0.1];
		\node [below] at (6,1) {0 = $\sim$ 1};
		\node [below] at (6,.5) {0 = $\neg$ 1};
		\draw [fill] (6,3) circle [radius = 0.1];
		\node [above] at (6,3) {1 = $\sim$ 0};
		\node [above] at (6,3.5) {1 = $\neg$ 0};
		
		\node [right] at (7,2) {$\textbf{3}$ := };
		\draw (9,0) -- (9,2);
		\node [below] at (9,0) {0 = $\sim$ 1 = $\neg$  1 };
		\node [below] at (9.25,-0.5) { = $\sim u$  };
    \draw [fill] (9,0) circle [radius = 0.1];
		\node [right] at (9,2) {$u$};
    \draw [fill] (9,2) circle [radius = 0.1];
		\draw (9,2) -- (9,4);
		\draw [fill] (9,4) circle [radius = 0.1];
		\node [above] at (9,4) {1 = $\sim$ 0 = $\neg$ 0};
		\node [above] at (9.25,4.5) { = $\neg u$};

\node [right] at (11,2) {$\textbf{4}$ := };
		\draw (13,0) -- (13,2);
		\node [below] at (13,0) {0 = $\sim$ 1 = $\neg$  1  };
		\node [below] at (13,-0.5) { = $\sim u_{1}  = \sim u_{2}$ };
    \draw [fill] (13,0) circle [radius = 0.1];
		\node [right] at (13,1.25) {$u_{2}$};
    \draw [fill] (13,1.25) circle [radius = 0.1];
		\node [right] at (13,2.75) {$u_{1} $};
		\draw [fill] (13,2.75) circle [radius = 0.1];
		\draw (13,2) -- (13,4);
		\draw [fill] (13,4) circle [radius = 0.1];
		\node [above] at (13,4) {1 = $\sim$ 0 = $\neg$ 0};
		\node [above] at (13,4.5) {= $\neg u_{1}  = \neg u_{2}$};

\end{tikzpicture}
\caption{Subdirectly irreducible double Stone algebras}

\label{fig5}
\end{figure}

Hence, again using Birkhoff's representation theorem we have the following.
\begin{theorem}\label{thm12A}  \cite{Balbes74}

Let $\mathcal{DS} = (DS,\vee,\wedge,\sim,\neg,0,1)$ be a double Stone algebra. There exists a (index) set $I$ such that $\mathcal{DS}$ can be 
embedded into {\rm$\textbf{4}^{I}$}.
\end{theorem}
\noindent Hence in particular given a Boolean algebra $B$,  the double Stone algebra $B^{[3]}$ is embeddable into $4^{I}$ for index set $I$. Now, following similar idea as above, let us characterize completely join irreducible elements of $B^{[3]}$ and $4^{I}$. Let us denote by \\
\hspace*{2 cm}$f_{i}^{x}(k)  :=
\left\{
	\begin{array}{ll}
		x  & \mbox{if } k = i \\
		0 & otherwise
	\end{array}
\right.$  
\begin{proposition}
\begin{enumerate}
\item The set of completely join irreducible elements of the algebra {\rm$\textbf{4}^{I}$} is given by\\
\hspace*{2cm} $\mathcal{J}_{\textbf{4}^{I}} = \{f^{a}_{i},f^{b}_{i},f^{1}_{i}: i \in I\}$.\\
\noindent Moreover, $\mathcal{J}_{\textbf{4}^{I}}$ is join dense in $\textbf{4}^{I}$.
\item  Let $B$ be a complete atomic Boolean algebra. The set of completely join irreducible  elements  of $B^{[3]}$ is given by\\
\hspace*{2 cm} $\mathcal{J}_{B^{[3]}} = \{(0,0,a),(0,a,a),(a,a,a): a \in \mathcal{J}_{B}\}$.\vskip 2pt 
\noindent Moreover, $\mathcal{J}_{B^{[3]}}$ is join dense in $B^{[3]}$.
\end{enumerate}
\end{proposition}
\begin{proof}
Proof of this proposition is very similar to the Proposition \ref{pp2} and can be seen \cite{KB17}.
\end{proof}
\noindent The order structure of $\mathcal{J}_{\textbf{4}^{I}}$ and $\mathcal{J}_{(\textbf{2}^{I})^{[3]}}$ can be visualized in the Figure \ref{fig6}.
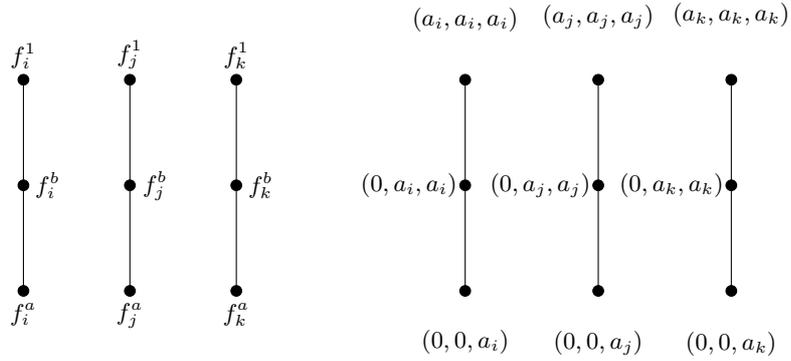
\begin{figure}[h] 
\begin{tikzpicture}[scale=.7]
    \tikzstyle{every node}=[draw,circle,fill=black,minimum size=4pt,
                            inner sep=0pt]

    \draw [dotted] (0,1) -- (0,1);
    
    \draw (3.2,0) -- (3.2,2);
		\draw (3.2,0) node [label=below:{\it $f_{i}^{a}$}]{};
		\draw (3.2,2) node [label=right:{\it $f_{i}^{b}$}]{};
		\draw (3.2,2) -- (3.2,4);
		\draw (3.2,4) node [label=above:{\it $f_{i}^{1}$}]{};
		\draw (5.2,0) -- (5.2,2);
		\draw (5.2,0) node [label=below:{\it $f_{j}^{a}$}]{};
		\draw (5.2,2) node [label=right:{\it $f_{j}^{b}$}]{};
		\draw (5.2,2) -- (5.2,4);
		\draw (5.2,4) node [label=above:{\it $f_{j}^{1}$}]{};
		\draw (7.2,0) -- (7.2,2);
		\draw (7.2,0) node [label=below:{\it $f_{k}^{a}$}]{};
		\draw (7.2,2) node [label=right:{\it $f_{k}^{b}$}]{};
		\draw (7.2,2) -- (7.2,4);
		\draw (7.2,4) node [label=above:{\it $f_{k}^{1}$}]{};
  \draw (11.5,0) -- (11.5,2);
		\draw (11.5,0) node [label=below:{\it $(0,0,a_{i})$}]{};
		\draw (11.5,2) node [label=left:{\it $(0,a_{i},a_{i})$}]{};
		\draw (11.5,2) -- (11.5,4);
		\draw (11.5,4) node [label=above:{\it $(a_{i},a_{i},a_{i})$}]{};
		\draw (14,0) -- (14,2);
		\draw (14,0) node [label=below:{\it $(0,0,a_{j})$}]{};
		\draw (14,2) node [label=left:{\it $(0,a_{j},a_{j})$}]{};
		\draw (14,2) -- (14,4);
		\draw (14,4) node [label=above:{\it $(a_{j},a_{j},a_{j})$}]{};
		\draw (16.5,0) -- (16.5,2);
		\draw (16.5,0) node [label=below:{\it $(0,0,a_{k})$}]{};
		\draw (16.5,2) node [label=left:{\it $(0,a_{k},a_{k})$}]{};
		\draw (16.5,2) -- (16.5,4);
		\draw (16.5,4) node [label=above:{\it $(a_{k},a_{k},a_{k})$}]{};
\end{tikzpicture}
\caption{Hasse diagram of $\mathcal{J}_{\textbf{4}^{I}}$}
\label{fig6}
\end{figure}

 
%

%
  %
    %
		%
\begin{theorem}\label{thm12} 
\blr
 \item The  sets of completely join irreducible elements of 
 {\rm $\textbf{4}^{I}$} and {\rm $(\textbf{2}^{I})^{[3]}$} are order isomorphic.
The  algebras {\rm $\textbf{4}^{I}$} and {\rm $(\textbf{2}^{I})^{[3]}$}
are lattice isomorphic.
\item The double Stone algebras {\rm $\textbf{4}^{I}$} and {\rm $(\textbf{2}^{I})^{[3]}$}
are  isomorphic.
\item Let $\mathcal{DBS} $ be a double Stone  algebra. There exists a (index) set $I$ such that $\mathcal{DBS}$ can be 
embedded into Stone algebra {\rm $(\textbf{2}^{I})^{[3]}$}.

\elr
\end{theorem}
\begin{proof}
We define the map 
{\rm $\phi: \mathcal{J}_{\textbf{4}^{I}} \rightarrow \mathcal{J}_{(\textbf{2}^{I})^{[3]}}$} as follows. For $i \in I$,\\
\hspace*{3 cm} $\phi(f_{i}^{u_{2}}) := (0,0,g_{i}^{1})$,\\
\hspace*{3 cm} $\phi(f_{i}^{u_{1}}) := (0,g_{i}^{1},g_{i}^{1})$,\\
\hspace*{3 cm} $\phi(f_{i}^{1}) := (g_{i}^{1},g_{i}^{1},g_{i}^{1})$.\\
Here $g_{i}^{1}$ is an atom of the Boolean algebra $\textbf{2}^{I}$, defined as follows:\\
\hspace*{3 cm}$g_{i}^{1}(k)  :=
\left\{
	\begin{array}{ll}
		1  & \mbox{if } k = i \\
		0 & otherwise
	\end{array}
\right.$

\noindent It can be easily seen that $\phi$ is an order isomorphism. Hence using Lemma \ref{lemma1},  
$\phi$ can be extended to {\rm $\textbf{4}^{I}$} as $\Phi$, then $\Phi$ is a lattice isomorphism from  {\rm $\textbf{4}^{I}$} to {\rm $(\textbf{2}^{I})^{[3]}$}.  Now, as, double Stone algebras are equational algebras, so the extended map of $\Phi$ is also isomorphic as double Stone algebras. This proves $(i)$ and $(ii)$. $(iii)$ follows from $(ii)$ and \ref{thm12A}.
\end{proof}
\noindent Now, let illustrate this Theorem through an example.
\begin{example}\label{ex1}
Let us consider a lattice whose Hasse diagram is given in the Figure \ref{fig6A}. Note that $\sim 1 = \sim a  = \sim b = \sim c = \sim d = 0 = \neg 1$ and $\neg 0 = \neg a  = \neg b = \neg c = \neg d = 1 = \sim 1$. It can be easily verified that  the given lattice is also a double Stone algebra. The given double Stone algebra is isomorphic (Figure \ref{fig6A}) to a subalgebra double Stone algebra $(\textbf{2}^{2})^{[2]}$. Here $\textbf{2}^{2}$ is the 4-element Boolean algebra $\{0,x,y,1\}$. 
\begin{figure}[h] 
\begin{tikzpicture}[scale=.7]
    \tikzstyle{every node}=[draw,circle,fill=black,minimum size=4pt,
                            inner sep=0pt]

    \draw [dotted] (0,1) -- (0,1);

		\draw (3.2,0) -- (3.2,1.5);
		\draw (3.2,0) node [label=below:{0}]{};
		\draw (3.2,1.5) node [label=right:{a}]{};
		\draw (3.2,5) -- (3.2,6.5);
		\draw (3.2,5) node [label=right:{d}]{};
		\draw (3.2,6.5) node [label=above:{1}]{};
		\draw (3.2,5) -- (1,3.25);
		\draw (1,3.25) node [label=left:{b}]{};
		\draw (3.2,1.5) -- (1,3.25);
		\draw (3.2,1.5) -- (5.4,3.25);
		\draw (3.2,5) -- (5.4,3.25);
    \draw (5.4,3.25) node [label=right:{c}]{};
\draw (11,0) -- (11,1.5);
		\draw (11,0) node [label=below:{$(0,0,0)$}]{};
		\draw (11,1.5) node [label=right:{$(0,y,y)$}]{};
		\draw (11,5) -- (11,6.5);
		\draw (11,5) node [label=right:{$(y,y,1)$}]{};
		\draw (11,6.5) node [label=above:{$(1,1,1)$}]{};
		\draw (11,5) -- (8.8,3.25);
		\draw (8.8,3.25) node [label=left:{$(0,y,1)$}]{};
		\draw (11,1.5) -- (8.8,3.25);
		\draw (11,1.5) -- (13.2,3.25);
		\draw (11,5) -- (13.2,3.25);
    \draw (13.2,3.25) node [label=right:{$(y,y,y)$}]{};		
		
\end{tikzpicture}
\caption{Hasse diagram of isomorphism in Example \ref{ex1}}
\label{fig6A}
\end{figure}
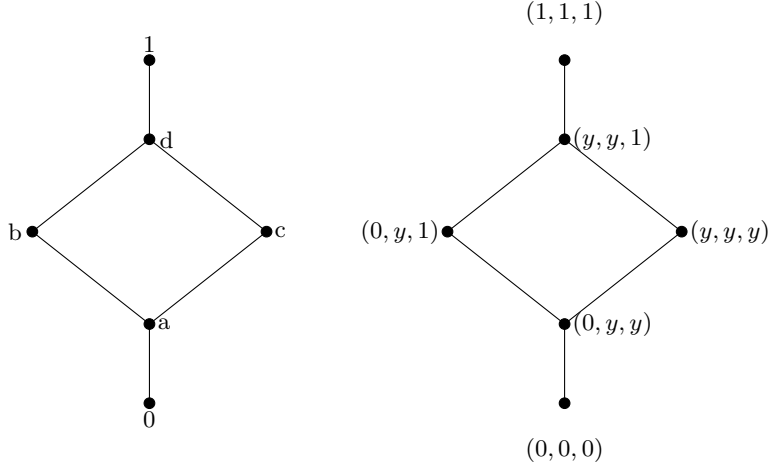

\end{example}
We end this section by providing an easy consequence of Stone's representation and Theorem \ref{thm12}.  
\begin{theorem}\label{thm13}
Given a double Stone algebra $\mathcal{DBS} = (DBS,\vee,\wedge,\sim,\neg,0,1)$, then there exists a set $U$ such that $\mathcal{DBS}$ can be embedded into double Stone algebra formed by the set $(\mathcal{P}(U))^{[3]}$. 
\end{theorem}
\subsection{4-valued semantics of logic for double Stone  algebras}
%
The idea of arriving at 4-valued semantics of $\mathcal{L}_{\mathcal{DBS}}$ is to represent an element of a given double Stone algebra as a tuple $(A,B,C)$, where $A\subseteq B \subseteq C$ are subsets of some set $U$. This proposition is established by Theorem \ref{thm13}.
So for any $x \in U$, there are 4 possibilities. 

\noindent \hspace{3 cm} $x \in A$,  $x \in B\setminus A$, $x \in C\setminus B$, and $x \notin C$.\\
\noindent These 4 possibilities leads to the 4-valued semantics of the  logic $\mathcal{L_{\mathcal{DBS}}}$.  
Now, let us formally present the concerned logic $\mathcal{L_{\mathcal{DBS}}}$. Let $\mathcal{F}_{\sim,\neg}$ denote the set of formulae  by extending the syntax of   $BDLL$ by adding  unary connectives $\sim, \neg$. The logic  $\mathcal{L_{\mathcal{DBS}}}$ denote the logic $BDLL$ along with  following rules and postulates.
\begin{enumerate} 
\item $\sim \alpha \wedge \sim \beta \vdash \sim(\alpha \vee \beta)$, 
 $\neg(\alpha \wedge \beta) \vdash \neg\alpha \vee \neg\beta$.
\item $\top \vdash \sim \bot$, 
$\neg\top \vdash  \bot$.
\item $\alpha \vdash \beta / \sim\beta \vdash  \sim\alpha$, 
$\alpha \vdash \beta / \neg\beta \vdash  \neg\alpha$.
\item $\sim \alpha \wedge \neg \beta \vdash \neg (\alpha \vee \beta)$, 
 $\sim (\alpha \wedge \beta) \vdash \sim\alpha \vee \neg\beta$.
			\item $\alpha ~\wedge \sim \alpha \vdash \bot$, $\alpha \vee \neg \alpha \vdash \bot$.
			\item $\alpha \wedge \beta \vdash \gamma /\alpha~ \wedge \sim \gamma \leq \sim \beta$,  $\gamma \vdash \alpha \vee \beta/ \neg \beta \vdash \alpha \vee \neg \gamma$.
			\item $\top \vdash \sim \alpha~ \vee \sim\sim \alpha $, $\neg \alpha \wedge \neg\neg \alpha \vdash \bot$.

			\end{enumerate}
With usual constructions of Lindenbaum–Tarski algebra and an easy consequence of Theorem \ref{thm13} is.			
\begin{theorem}
The followings are equivalent:
\begin{enumerate}
\item $\alpha \vdash_{\mathcal{L}_{DBS}} \beta$.  
\item $\alpha \vDash_{\mathcal{A}_{DBS}} \beta$.
\item $\alpha \vDash_{\mathcal{DBS}B^{[2]}} \beta$.
\item $\alpha \vDash_{\mathcal{DBS}(\mathcal{P}(U))^{[3]}} \beta$.

\end{enumerate}
\end{theorem}	
\begin{theorem}\label{thm15}
$\alpha \vDash_{\mathcal{DBS}_{(\mathcal{P}(U))^{[3]}}} \beta$ if and only if $\alpha \vDash_{4} \beta$, for any $\alpha,\beta \in \mathcal{F}$.
\end{theorem} 
\begin{proof}
Let $\alpha \vDash_{\mathcal{DBS}_{(\mathcal{P}(U))^{[3]}}} \beta$, and
 $v: \mathcal{F} \rightarrow \textbf{4}$ be a valuation. By Theorem \ref{thm13},
$\textbf{4}$ is isomorphic to a double Stone algebra of $\mathcal{P}(U)^{[3]}$ for some set $U$. If this isomorphism is denoted by $\phi$,  $\phi \circ v$ is  a valuation in  $\mathcal{P}(U)^{[3]}$. Then $(\phi \circ v)(\alpha) \leq (\phi \circ v)(\beta)$ implies $v(\alpha) \leq  v(\beta)$. \vskip 2pt

\noindent Now, let $\alpha \vDash_{4} \beta$. Let $U$ be a set, and $\mathcal{P}(U)^{[3]}$ be the corresponding double Stone algebra.
  Let $v$ be a valuation on $\mathcal{P}(U)^{[3]}$ -- we need to show $v(\alpha) \leq v(\beta)$. For any $\gamma \in \mathcal{F}$ with $v(\gamma) := (A,B,C),~A,B,C \subseteq U,$ and for each $x \in U$, define a map $v_{x}: \mathcal{F} \rightarrow \textbf{4}$ as\\
\hspace*{2 cm}$v_{x}(\gamma)  :=
\left\{
	\begin{array}{lll}
		t & \mbox{if }   x \in A \\
		u_{1} & \mbox{if } x \in B  \setminus A  \\
		u_{2} & \mbox{if } x \in C  \setminus B  \\
		f & \mbox{if } x \notin C .
	\end{array}
\right.$ \vskip 2pt 
\noindent  Consider any $\gamma, \delta \in \mathcal{F}$, with  $v(\gamma) := (A,B,C)$ and $v(\delta) := (D,E,F),~A,B,C,D,E,F \subseteq U$. It is easy to show that: $v_{x}(\gamma \wedge \delta) = v_{x}(\gamma) \wedge v_{x}(\delta)$, $v_{x}(\gamma \vee \delta) = v_{x}(\gamma) \vee v_{x}(\delta)$.
Let us show the followings. 
\begin{itemize} \item
$v_{x}(\sim\gamma) = \sim v_{x}(\gamma)$.

\noindent Note that $ v(\sim \gamma)= (C^{c},C^{c},C^{c})$.\\
\noindent \un{Case 1} $v_{x}(\gamma) = t$: Then $x \in A$, i.e. $x \notin A^{c}$ and so $x \notin C^{c}$. Hence 
          $v_{x}(\sim\gamma)  = f = \sim v_{x}(\gamma)$.\\
\noindent \un{Case 2} $v_{x}(\gamma) = u_{1}$:  $x \notin A$ but $x \in B$, i.e.,  $x \notin B^{c}$ so, $x \notin C^{c}$. Hence 
          $v_{x}(\sim\gamma)  = f = \sim v_{x}(\gamma)$.\\
\noindent \un{Case 3} $v_{x}(\gamma) = u_{2}$:  $x \notin B$ but $x \in C$, i.e.,  $x \notin C^{c}$. Hence 
          $v_{x}(\sim\gamma)  = f = \sim v_{x}(\gamma)$.\\						
\noindent \un{Case 4} $v_{x}(\gamma) = f$:  $x \notin C$, i.e. $x \in C^{c}$. So 
          $v_{x}(\sim\gamma)  = t = \sim v_{x}(\gamma)$.
\item $v_{x}(\neg\gamma) = \neg v_{x}(\gamma)$.		

\noindent Note that $ v(\neg \gamma)= (A^{c},A^{c},A^{c})$.\\
\noindent \un{Case 1} $v_{x}(\gamma) = t$: Then $x \in A$, i.e. $x \notin A^{c}$.  Hence 
          $v_{x}(\neg\gamma)  = f = \neg v_{x}(\gamma)$.\\
\noindent \un{Case 2} $v_{x}(\gamma) = u_{1}$:  $x \notin A$ but $x \in B$, i.e.,  $x \in A^{c}$.  Hence 
          $v_{x}(\neg\gamma)  = t = \neg v_{x}(\gamma)$.\\
\noindent \un{Case 3} $v_{x}(\gamma) = u_{2}$:  $x \notin B$ but $x \in C$, i.e.,  $x \in B^{c}$ so $x \in A^{c}$. Hence 
          $v_{x}(\neg\gamma)  = t = \neg v_{x}(\gamma)$.\\						
\noindent \un{Case 4} $v_{x}(\gamma) = f$:  $x \notin C$, i.e. $x \in C^{c}$. So 
          $v_{x}(\neg\gamma)  = t = \neg v_{x}(\gamma)$.												
														
\end{itemize}

\noindent Hence $v_{x}$ is a valuation in $\textbf{4}$. Now let us show that $v(\alpha) \leq v(\beta)$. 
Let  $v(\alpha) := (A^{\prime},B^{\prime},C^{\prime})$, 
$v(\beta) := (D^{\prime},E^{\prime},F^{\prime})$. 

Let $x \in A^{\prime}$ then $v_{x}(\alpha) = t$. As $\alpha \vDash_{4} \beta$ hence $v_{x}(\beta) = t$. This implies $x \in D^{\prime}$, whence $A^{\prime} \subseteq D^{\prime}$.

Let $x \in B^{\prime}\setminus A^{\prime}$ then $v_{x}(\alpha) = u_{1}$. As $\alpha \vDash_{4} \beta$ hence $v_{x}(\beta) = t$ or $u_{1}$. This implies $x \in E^{\prime}$, whence $B^{\prime} \subseteq E^{\prime}$.


On the other hand, if $x \notin F^{\prime}$, $v_{x}(\beta) = f$. By assumption, we have   $v_{x}(\alpha) = f$, so that $x \notin C^{\prime}$, giving  $C^{\prime} \subseteq F^{\prime}$.
\end{proof}\qed
Finally, we have the following 4-valued semantics of the logic $\mathcal{L}_{DBS}$.
\begin{theorem}(4-valued semantics)
 $\alpha \vdash_{\mathcal{L}_{DBS}} \beta$ if and only if $\alpha \vDash_{\textbf{4}} \beta$.
\end{theorem}	
Now, a natural question arises here, what happens if we follow the Dunn's semantic consequence relation of De Morgan logic?
%
%
%
%
\begin{definition}
 Let $\alpha,\beta \in \mathcal{F}_{DBS}$.\vskip 2pt
\blr
\item   $\alpha \vDash_{1} \beta$ if and only  if, if  $v(\alpha) = 1$ then  $v(\beta) = 1$~~ {\rm (Truth preservation)}.
\item  $\alpha \vDash_{0} \beta$ if and only if, if $v(\beta) = 0$ then  $v(\alpha) = 0$  {\rm (Falsity preservation)}.
\item  $\alpha \vDash_{0,1} \beta$ if and only if,  $\alpha \vDash^{S}_{1} \beta$ {\it and} $\alpha \vDash_{0} \beta$.
\elr
\end{definition} 

\begin{lemma}

$\alpha \vDash_{1} \beta$ if and only if $\alpha \vDash_{0} \beta$. 

\end{lemma}
\begin{proof}
 We prove `if' part, proof of `only if' part follows similarly. 

 Assume $\alpha \vDash_{1} \beta$ and let $v: \mathcal{F} \rightarrow \textbf{4}$ such that $v(\beta) = 0$. 
As, $\alpha \vDash_{1} \beta$ so $v(\alpha) \neq 1$. Assume $v(\alpha) = u_{1}$ and define $v_{1}: \mathcal{F} \rightarrow \textbf{4}$ as $v_{1}(\gamma) = \sim\sim v(\gamma)$. Then $v_{1}$ is a valuation in $\textbf{4}$ and $v_{1}(\alpha) = 1$. So, $v_{1}(\beta) = 1$ but $v_{1}(\beta) = v(\beta) = 0$ which is absurd. Hence $v(\alpha) \neq u_{1}$. Similarly  $v(\alpha) \neq u_{2}$.

\end{proof}\qed
\begin{proposition}

If $\alpha \vDash_{\textbf{4}} \beta$ then $\alpha \vDash_{1,0} \beta$.  
\end{proposition}
Let us show through an example below that converse of the above proposition may not be true.
\begin{example}
Consider the sequent $\alpha \wedge \neg \alpha \vdash \beta \vee \sim \beta$. Let us show that $\alpha \wedge \neg \alpha \vDash_{1,0} \beta \vee \sim \beta$. So, assume that $v$ is a valuation in $\textbf{4}$ such that $v(\alpha \wedge \neg \alpha) = 1$. So, $v(\alpha) = 1$ and $\neg v(\alpha) = 1$. But there is no $x \in \textbf{4}$ such that $x \wedge \neg x = 1$. Hence $\alpha \wedge \neg \alpha \vDash_{1} \beta \vee \sim \beta$ is vacuously true. 
Similarly, there is no $x \in \textbf{4}$ such that $x \vee \sim x = 0$. Hence $\alpha \wedge \neg \alpha \vDash_{0} \beta \vee \sim \beta$ is vacuous.

So, $\alpha \wedge \neg \alpha \vDash_{1,0} \beta \vee \sim \beta$. 

\noindent Now let us show that $\alpha \wedge \neg \alpha \nvDash_{\textbf{4}} \beta \vee \sim \beta$. Define a valuation $v$ in $\textbf{4}$ such that $v(\alpha) = u_{1}$ and $v(\beta) = u_{2}$. Then $v(\alpha \wedge \neg \alpha) = v(\alpha) \wedge \neg v(\alpha) = u_{1} \wedge 1 = u_{1}$ and $v(\beta \vee \sim \beta) = u_{2}$. But $u_{1} \nleq u_{2}$, hence $\alpha \wedge \neg \alpha \nvDash_{\textbf{4}} \beta \vee \sim \beta$.
\end{example} 

\section{Conclusions}\label{sec4}
The lattices $B^{[2]}$ and $B^{[3]}$ can be extended to form various algebraic structures. In this article, we have studied  
$B^{[2]}$ as (dual) Stone algebra and $B^{[3]}$ as double Stone algebra.  Moisil obtained representation of 3-valued and 4-valued  {\L}ukasiewicz algebra  in terms of  $B^{[2]}$ and $B^{[3]}$ respectively. Similar to his works we have obtained the representations of (dual) Stone and double Stone algebras in terms of  $B^{[2]}$ and $B^{[3]}$ respectively.

The $\textbf{2}$ element Boolean algebra play a fundamental role in classical propositional logic (via $\textbf{T}$rue- $\textbf{F}$alse) semantics and Boolean algebra (via Stone's representation theorem). We have established the same for:
\begin{enumerate}
\item The Stone algebra $\textbf{3}_{\sim}$, the logic $\mathcal{L}_{\mathcal{S}}$ and the class of all Stone algebras.
\item The Stone algebra $\textbf{3}_{\neg}$, the logic $\mathcal{L}_{\mathcal{DS}}$ and the class of all dual Stone algebras.
\item The double Stone algebra $\textbf{4}$, the logic $\mathcal{L}_{\mathcal{DBS}}$ and the class of all double Stone algebras.
\end{enumerate}
\noindent Moreover, due to rough set representation result  of (dual) Stone algebra, we have been able to provide rough set semantics of the logic ($\mathcal{L}_{\mathcal{DS}}$) $\mathcal{L}_{\mathcal{S}}$ which thus leads to the equivalence of 3-valued, algebraic and rough set semantics of the logic ($\mathcal{L}_{\mathcal{DS}}$) $\mathcal{L}_{\mathcal{S}}$.       
\end{document}